\documentclass[10pt,a4paper]{article}
\usepackage{amsmath,amsfonts,amssymb,amsthm,amscd}
\newtheorem{thm}{Theorem}
\newtheorem{lem}[thm]{Lemma}
\newtheorem{prop}[thm]{Proposition}
\newtheorem{defn}[thm]{Definition}

\newtheorem{ex}[thm]{Example}
\newtheorem{rem}[thm]{Remark}
\newtheorem{notations}[thm]{Conventions}

\begin{document}

\title{Dubrovin's duality for $F$-manifolds with eventual identities}
\date{}

\author{Liana David and Ian A.B. Strachan}
\maketitle

{\bf Abstract:} A vector field $\mathcal E$ on an $F$-manifold
$(M, \circ , e)$ is an eventual identity if it is invertible and
the multiplication $X* Y := X\circ Y\circ {\mathcal E}^{-1}$
defines a new $F$-manifold structure on $M$. We give a characterization of such
eventual identities, this being a problem raised by Manin
\cite{manin}. We develop a duality between $F$-manifolds with
eventual identities and we show that is compatible with the local
irreducible decomposition of $F$-manifolds and preserves the class of
Riemannian $F$-manifolds. We find necessary and sufficient conditions
on the eventual identity which insure that harmonic Higgs bundles
and $DChk$-structures are preserved by our duality.
We use eventual identities to construct
compatible pair of metrics.

\section{Introduction}

In \cite{dubrovin} Dubrovin introduced the idea of an almost dual Frobenius manifold.
Starting from a Frobenius manifold one may construct a new geometric object that shares
many, but crucially not all, of the essential features of the original manifold. In particular
a new \lq dual\rq~solution of the underlying Witten-Dijkgraaf-Verlinde-Verlinde (WDVV) equations
may be constructed from the original manifold.  Such a construction reflects certain other
\lq dualities\rq~that occur in other areas of mathematics where Frobenius manifolds appear.
For example, in:

\begin{itemize}

\item{} Quantum cohomology and mirror symmetry;

\item{} Integrable systems, via generalizations of the classical Miura transform;

\item{} Singularity theory, via the correspondence between oscillatory integrals and period integrals.

\end{itemize}

\noindent More specifically, given a Frobenius manifold
$(M\,,\circ\,,e\,,E\,,{\tilde g})$ with multiplication $\circ\,,$
unity field $e\,,$ Euler field $E$ and metric ${\tilde g}$ one may
define a new multiplication $*$ and metric $g$ by the formulae
\begin{eqnarray*}
X* Y & = & X \circ Y \circ E^{-1}\,,\\
g(X,Y) & = & {\tilde g}(E^{-1} \circ X, Y)
\end{eqnarray*}
where $E^{-1} \circ E = e\,.$ Clearly $*$ is associative,
commutative and has a unity, namely $E\,,$ the original Euler
field. The new metric $g$ (the intersection form) turns out to be
flat and from these two new objects one may define a dual solution
to the WDVV-equations. This correspondence is not completely dual
- certain properties are lost. For example, while ${\widetilde
\nabla} e=0\,,$ the new identity does not share this property: in
general $\nabla E \neq 0\,.$

Underlying Frobenius manifolds is a structure known as an $F$-manifold, which was introduced
by Hertling and Manin \cite{HM}.

\begin{defn}\cite{HM} i)  An $F$-manifold is a triple $(M, \circ , e)$ where $M$
is a manifold, $\circ$ is a
commutative, associative multiplication on the tangent bundle
$TM$, with identity vector field $e$, such that the $F$-manifold
condition
\begin{equation}\label{fman}
L_{X\circ Y} (\circ ) := X\circ L_{Y}(\circ ) + Y\circ L_{X}(\circ
),
\end{equation}
holds, for any smooth vector fields $X, Y\in {\mathcal X}(M).$

\medskip

ii) An Euler vector field (of weight $d$) on an $F$-manifold $(M,
\circ , e)$ is a vector field $E$ which preserves the
multiplication up to a constant, i.e.
$$
L_{E}(\circ )(X, Y) = d\,X \circ Y,\quad\forall X, Y\in {\mathcal
X}(M).
$$
\end{defn}

\noindent $F$-manifolds appear in many areas of mathematics. All Frobenius manifolds
have an underlying $F$-manifold structure, and in examples originating from singularity
theory such $F$-manifolds arise in a very natural way \cite{hert}. They also appear
within integrable systems - both in examples coming from the submanifold geometry of
Frobenius manifolds \cite{iabs1} and non-local biHamiltonian geometry \cite{l-i} and
their role has been elucidated further in \cite{ita}.

Given an $F$-manifold with an Euler vector field one may construct
a dual multiplication via $X* Y=X \circ Y \circ E^{-1}\,.$ While
this is commutative and associative with unity element, whether or
not this defines an $F$-manifold is not immediately clear. More
generally, Manin \cite{manin} replaced the Euler field $E$ by an
arbitrary invertible vector field and used this to define a new
multiplication.

\begin{defn}\cite{manin} A vector field $\mathcal E$ on an
$F$-manifold $(M, \circ , e)$ is called an eventual identity, if
it is invertible (i.e. there is a vector field ${\mathcal
E}^{-1}$ such that ${\mathcal E}\circ {\mathcal E}^{-1}={\mathcal
E}^{-1}\circ {\mathcal E} =e$) and, moreover, the multiplication
\begin{equation}\label{star}
X*Y = X\circ Y\circ {\mathcal E}^{-1},\quad\forall X, Y\in {\mathcal X}(M)
\end{equation}
defines a new $F$-manifold structure on $M$.
\end{defn}

\noindent The reason for the
terminology is that ${\mathcal E}$ is the identity vector field
for the multiplication $*$. In this paper we give the characterization
of such eventual identities, thus answering a question raised by
Manin \cite{manin}.

\begin{thm}\label{main} i) Let $(M, \circ , e)$ be an $F$-manifold and ${\mathcal E}$
an invertible vector field. Then ${\mathcal E}$ is an eventual identity
if and only if
\begin{equation}\label{char}
L_{\mathcal E}(\circ ) ( X, Y)= [e, {\mathcal E}]\circ X\circ Y,\quad\forall
X, Y\in {\mathcal X}(M).
\end{equation}
ii) Let
$$
X*Y = X\circ Y\circ {\mathcal E}^{-1}
$$
be the new $F$-manifold multiplication. Then the map
$$
(M, \circ , e, {\mathcal E} )\rightarrow (M, * , {\mathcal E}, e)
$$
is an isomorphism between $F$-manifolds with eventual identities.
\end{thm}
Condition (\ref{char}) above may be seen as a generalization of
the notion of an Euler vector field. All invertible Euler vector
fields are eventual identities but not conversely. However,
eventual identities play a similar role. In this paper we study
$F$-manifolds with eventual identities and their relation with
some well-known constructions in the theory of Frobenius
manifolds.

The plan of the paper is the following. In Section
\ref{mainsection} we prove Theorem \ref{main} and we develop its
consequences. We remark that the duality for $F$-manifolds with
eventual identities developed in Theorem \ref{main} {\it ii)} is a
natural generalization of the well-known dualities for almost
Frobenius manifolds and for $F$-manifolds with compatible flat
structures \cite{dubrovin}, \cite{manin}. After proving Theorem
\ref{main} we show that any eventual identity on a product
$F$-manifold is a sum of eventual identities on the factors (a
similar decomposition holds for Euler vector fields \cite{hert}).
Using this fact we show that our duality for $F$-manifolds with
eventual identities is compatible with the local irreducible
decomposition of $F$-manifolds \cite{hert}. We end Section
\ref{mainsection} with examples and further properties of eventual
identities, some of them being already known for Euler vector
fields.

In Section \ref{F-riem} we add a new ingredient on our
$F$-manifold $(M, \circ , e ,{\mathcal E})$ with eventual
identity, namely a multiplication invariant metric $\tilde{g}$.
The eventual identity $\mathcal E$ together with $\tilde{g}$
determine, in a canonical way, a second metric $g$, defined like
the second metric of a Frobenius manifold. We prove that the
metrics $(g,\tilde{g})$ are almost compatible. Our main result in
this Section states that $(g, \tilde{g})$ are compatible, when
$(M, \circ , e,\tilde{g})$ is an almost Riemannian $F$-manifold,
i.e. the coidentity $\epsilon \in\Omega^{1}(M)$, which is the
$1$-form dual to the identity $e$, is closed. Similar results
already appear in the literature \cite{l-i}, with Euler vector
fields instead of eventual identities.

In Section \ref{dual} we show that our duality for $F$-manifolds
with eventual identities preserves the class of Riemannian
$F$-manifolds, which are almost Riemannian $F$-manifolds
satisfying an additional curvature condition. Riemannian
$F$-manifolds were introduced and studied in \cite{ita} and are
closely related to the theory of integrable systems of
hydrodynamic type.

In Section \ref{integrable} we apply
our results to the theory of integrable systems.

In Section \ref{ttstar-sect} we study the interactions between
$tt^{*}$-geometry, introduced for the first time in \cite{cecotti},
and our duality of $F$-manifolds with eventual identities.
$tt^{*}$-geometry shares many properties in commun with
Frobenius manifolds, its main ingredients being a
metric, a Higgs field and a real structure (the latter not
being present in the theory of Frobenius manifolds).
One can combine $tt^{*}$-geometry with Frobenius manifold theory
giving rise to new structures (like CDV-structures, $DChk$-structures, etc)
satisfying some complicated compatibility conditions, but which are
very natural in examples coming from singularity theory.
It is in this context that $F$-manifolds appear in $tt^{*}$-geometry.
We determine necessary and sufficient
conditions on the eventual identity which insure that the
class of harmonic Higgs bundles and $DChk$-structures (i.e. harmonic
Higgs bundles with compatible real structure) is preserved
by our duality for $F$-manifolds with eventual identities.

\section{Eventual identities and duality}\label{mainsection}

In this Section we prove Theorem \ref{main}. We begin with a
simple preliminary Lemma concerning invertible vector fields on
$F$-manifolds.

\begin{lem}\label{pregatitoare} Let $(M, \circ , e)$ be an
$F$-manifold and ${\mathcal E}$
an invertible vector field, with inverse ${\mathcal E}^{-1}.$  Assume that
\begin{equation}\label{e1}
L_{{\mathcal E}}(\circ ) (X, Y) = [e, {\mathcal E}]\circ X\circ Y,
\quad\forall X, Y\in {\mathcal X}(M).
\end{equation}
Then also
\begin{equation}\label{e2}
L_{{\mathcal E}^{-1}}(\circ ) (X, Y) = [e, {\mathcal E}^{-1}]
\circ X\circ Y,\quad
\forall X, Y\in {\mathcal X}(M).
\end{equation}
\end{lem}

\begin{proof} The proof is a simple calculation.
Since $e = e\circ e$, the $F$-manifold condition
(\ref{fman}) with $X= Y :=e$ implies that
$L_{e}(\circ )=0$. Applying again (\ref{fman}) with
$X:= {\mathcal E}$ and $Y:= {\mathcal E}^{-1}$, we obtain:
$$
0= L_{{\mathcal E}\circ {\mathcal E}^{-1}}(\circ )
= {\mathcal E}\circ L_{{\mathcal E}^{-1}}(\circ ) +
{\mathcal E}^{-1}\circ L_{{\mathcal E}}(\circ ) .
$$
Combining this relation with (\ref{e1}) we get
$$
L_{{\mathcal E}^{-1}}(\circ )(X, Y) = {\mathcal E}^{-2}\circ
[{\mathcal E}, e]\circ X\circ Y,\quad\forall X, Y\in {\mathcal X}(M),
$$
where ${\mathcal E}^{-2}$ denotes ${\mathcal E}^{-1}\circ {\mathcal E}^{-1}.$
On the other hand,
$$
[e,{\mathcal E}]\circ {\mathcal E}^{-2} = \left(
L_{e}({\mathcal E})\circ {\mathcal E}^{-1}\right)\circ {\mathcal E}^{-1}
=\left( L_{e}(e) - {\mathcal E}\circ L_{e}({\mathcal E}^{-1})\right)
\circ {\mathcal E}^{-1} = [{\mathcal E}^{-1}, e]
$$
where we used $L_{e}(\circ ) =0.$ Our claim follows.
\end{proof}

Note that the construction of ${\mathcal E}^{-1}\,,$ whilst just linear algebra, requires the
inversion of a matrix, and hence ${\mathcal E}^{-1}$ is not defined at points of $M$ where a
certain determinant $\Sigma$ vanishes. Rather than defining a new manifold $M^\star\cong M \backslash \Sigma$
on which ${\mathcal E}^{-1}$ is defined we just assume that $M$ consists of points at which both $\mathcal E$ and
${\mathcal E}^{-1}$ are well defined.

After this preliminary result, we now prove Theorem \ref{main}
stated in the Introduction.\\

{\bf Proof of Theorem \ref{main}.} The multiplication $*$ is
commutative, associative, with identity
field $\mathcal E .$  Therefore $(M, *, {\mathcal E})$ is an
$F$-manifold if
and only if for any vector fields $Z, V\in {\mathcal X}(M)$,
\begin{equation}\label{dedem}
L_{Z*V}(*)(X, Y) = Z* L_{V}(*)(X, Y) + V* L_{Z}(*)(X, Y),\quad\forall X, Y
\in {\mathcal X}(M).
\end{equation}
We will show that (\ref{dedem}) is equivalent with (\ref{char}).
For this, we take the Lie derivative with respect to $Z$ of the relation
(\ref{star}). We get, by a straightforward computation,
\begin{equation}\label{rel1}
L_{Z}(*)(X, Y) = L_{Z}(\circ )({\mathcal E}^{-1}\circ X, Y)
+L_{Z}(\circ ) ({\mathcal E}^{-1}, X)\circ Y + [Z, {\mathcal E}^{-1}]\circ
X\circ Y.
\end{equation}
Using relation (\ref{rel1}) with $Z$ replaced by $Z*V = Z\circ
V\circ {\mathcal E}^{-1}$ and the $F$-manifold condition
(\ref{fman}) satisfied by the multiplication $\circ$, we get:
\begin{align*}
L_{Z*V} (*) (X, Y) &= {\mathcal E}^{-1} \circ Z\circ L_{V}(\circ )
({\mathcal E}^{-1}\circ X, Y) + {\mathcal E}^{-1}\circ V\circ L_{Z}(\circ )
({\mathcal E}^{-1}\circ X, Y) \\
& + Z\circ V \circ L_{{\mathcal E}^{-1}}(\circ )({\mathcal E}^{-1}\circ X, Y)
+{\mathcal E}^{-1}\circ Z\circ Y \circ L_{V}(\circ ) ({\mathcal E}^{-1}, X)\\
& + {\mathcal E}^{-1}\circ V\circ Y \circ L_{Z}(\circ ) ({\mathcal E}^{-1}, X)
+ Z\circ V\circ Y\circ L_{{\mathcal E}^{-1}}(\circ ) ({\mathcal E}^{-1}, X)\\
& - L_{{\mathcal E}^{-1}}({\mathcal E}^{-1}\circ Z\circ V)\circ X\circ Y.
\end{align*}
Combining this expression with the expressions of $L_{Z}(*)(X, Y)$
and $L_{V}(*)(X, Y)$ provided by (\ref{rel1}), we see that
(\ref{dedem}) holds if and only if
\begin{align*}
&X\circ Y\circ \left( L_{{\mathcal E}^{-1}}({\mathcal E}^{-1}\circ Z
\circ V) + {\mathcal E}^{-1} \circ Z\circ [V , {\mathcal E}^{-1}]
+{\mathcal E}^{-1}\circ V\circ [Z, {\mathcal E}^{-1}]\right) =\\
&Z\circ V\circ \left( L_{{\mathcal E}^{-1}}(\circ )
({\mathcal E}^{-1}\circ X, Y) +Y\circ L_{{\mathcal E}^{-1}}(\circ )
({\mathcal E}^{-1}, X)\right) .
\end{align*}
On the other hand, it can be checked that
\begin{align*}
&L_{{\mathcal E}^{-1}}({\mathcal E}^{-1}\circ Z
\circ V) + {\mathcal E}^{-1} \circ Z\circ [V , {\mathcal E}^{-1}]
+{\mathcal E}^{-1}\circ V\circ [Z, {\mathcal E}^{-1}]\\
&= L_{{\mathcal E}^{-1}}(\circ )({\mathcal E}^{-1}, Z)
\circ V + L_{{\mathcal E}^{-1}}
(\circ ) ({\mathcal E}^{-1}\circ Z, V) .
\end{align*}
Hence $*$ is the multiplication of an $F$-manifold
structure if and only if for any vector fields
$X, Y, Z, V\in {\mathcal X}(M)$,
\begin{align*}
&X\circ Y\circ\left( L_{{\mathcal E}^{-1}}(\circ )
({\mathcal E}^{-1}\circ Z, V)+L_{{\mathcal E}^{-1}}(\circ )
({\mathcal E}^{-1}, Z)\circ V\right) \\
&=Z\circ V\circ\left( L_{{\mathcal E}^{-1}}(\circ )
({\mathcal E}^{-1}\circ X, Y)
+L_{{\mathcal E}^{-1}}(\circ ) ({\mathcal E}^{-1}, X)\circ Y\right).\\
\end{align*}
Taking $X = Y := e$ it is easy to see that this relation
is equivalent with
\begin{equation}\label{ec}
L_{{\mathcal E}^{-1}}(\circ )({\mathcal E}^{-1}\circ Z, V)
+ L_{{\mathcal E}^{-1}}
(\circ ) ({\mathcal E}^{-1}, Z) \circ V = -2 {\mathcal E}^{-1}\circ
[{\mathcal E}^{-1}, e]\circ Z\circ V.
\end{equation}
We now simplify relation (\ref{ec}). For this, we take in
(\ref{ec}) $Z:= e$ and we obtain
\begin{equation}\label{ec1}
L_{{\mathcal E}^{-1}}(\circ )({\mathcal E}^{-1}, V) = - {\mathcal
E}^{-1} \circ [{\mathcal E}^{-1}, e]\circ V,\quad\forall V\in
{\mathcal X}(M).
\end{equation}
Combining (\ref{ec}) with (\ref{ec1}) we get:
\begin{equation}\label{ec2}
L_{{\mathcal E}^{-1}}(\circ ) (Z, V) = - [{\mathcal E}^{-1}, e]\circ
Z\circ V,\quad\forall Z, V\in {\mathcal X}(M).
\end{equation}
Conversely, it is clear that if (\ref{ec2}) is satisfied then (\ref{ec})
is satisfied as well. Therefore, relations (\ref{ec}) and (\ref{ec2})
are equivalent. We proved that  $*$ is the multiplication of an $F$-manifold
structure if and only if  (\ref{ec2}) holds. Our first claim follows
from Lemma \ref{pregatitoare}.

For our second claim, assume that $\mathcal E$ is an eventual
identity on an $F$-manifold $(M, \circ , e )$. We want to prove
that $e$ is an eventual identity for the $F$-manifold $(M, *,
{\mathcal E})$, where $*$ is related to $\circ$ by (\ref{star}).
Since the identity field of $*$ is $\mathcal E$, we need to show that
\begin{equation}
L_{e}(*) (X, Y) = [{\mathcal E}, e]* X*Y,\quad\forall X,
Y\in {\mathcal X}(M).
\end{equation}
Letting $Z:= e$ in (\ref{rel1}) and using $L_{e}(\circ )=0$
together with (\ref{star}), we get:
\begin{align*}
L_{e}(*) (X, Y) = [e, {\mathcal E}^{-1}]\circ X\circ Y
= \left( [e,{\mathcal E}^{-1}]\circ {\mathcal E}^{2}\right) * X*Y.
\end{align*}
Recall now from the
proof of Lemma \ref{pregatitoare} that $[e, {\mathcal
E}^{-1}]\circ {\mathcal E}^{2} =[{\mathcal E}, e].$ Our second
claim follows. The proof of Theorem \ref{main} is now completed.\\

Having found the characterization of eventual identities one may study how
such objects many be combined to form new eventual identities.

\begin{prop}\label{combine}{\it i)} Eventual identities form a subgroup of
the group of invertible vector fields on an $F$-manifold.

{\it ii)} The Lie bracket of two eventual identities is an
eventual identity, provided that is invertible.

{\it iii)} Let $(M_{1}\times M_{2}, \circ , e_{1}+e_{2})$ be the
product of two $F$-manifolds $(M_{1},\circ_{1}, e_{1})$ and
$(M_{2}, \circ_{2},e_{2})$, with multiplication defined by
\begin{equation}\label{product-mult}
(X_{1}, X_{2})\circ (Y_{1}, Y_{2}) = (X_{1}\circ_{1}
Y_{1},X_{2}\circ_{2} Y_{2}),
\end{equation}
for any $X_{1}, Y_{1}\in {\mathcal X}(M_{1})$ and $X_{2}, Y_{2}\in
{\mathcal X}(M_{2})$ (considered as vector fields on $M_{1}\times
M_{2}$). If ${\mathcal E}_{1}$ is an eventual identity on $(M,
\circ_{1}, e_{1})$ and $\mathcal E_{2}$ is an eventual identity on
$(M, \circ , e_{2})$, then $\mathcal E:= {\mathcal
E}_{1}+{\mathcal E}_{2}$ is an eventual identity on $(M_{1}\times
M_{2}, \circ , e_{1}+e_{2})$. Moreover, any eventual identity on
$(M_{1}\times M_{2}, \circ , e_{1}+e_{2})$ is obtained this way.

\end{prop}

\begin{proof} {\it i)} If ${\mathcal E}_{1}$ and ${\mathcal E}_{2}$
are eventual identities then ${\mathcal E}_{1}\circ {\mathcal
E}_{2}$ is invertible and for any $X, Y\in {\mathcal X}(M)$,
\begin{align*}
L_{{\mathcal E}_{1}\circ{\mathcal E}_{2}}(\circ )(X, Y)&=
{\mathcal E}_{1}\circ L_{{\mathcal E}_{2}}(\circ )(X, Y) +
{\mathcal E}_{2}\circ L_{{\mathcal E}_{1}}(\circ )(X, Y)\\
&= \left({\mathcal E}_{1}\circ [e, {\mathcal E}_{2}]+
{\mathcal E}_{2}\circ [e,{\mathcal E}_{1}]\right)\circ X\circ Y\\
&= [e, {\mathcal E}_{1}\circ {\mathcal E}_{2}]\circ X\circ Y
\end{align*}
where in the last equality we used $L_{e}(\circ )=0.$
Moreover, from Lemma \ref{pregatitoare} and Theorem \ref{main},
if $\mathcal E$ is an eventual identity then also ${\mathcal E}^{-1}$ is
an eventual identity. Our first claim follows.

{\it ii)} Recall the following relation proved in Proposition
4.3 of \cite{HM}: for any vector fields $X, Y, Z, W\in {\mathcal X}(M)$,
\begin{align*}
L_{[X, Y]}(\circ )(Z, W) = [X, L_{Y}(\circ )(Z, W)] -L_{Y}(\circ )([X, Z],W)
-L_{Y}(\circ )(Z, [X,W])\\
-[Y, L_{X}(\circ )(Z,W)]+L_{X}(\circ )([Y,Z], W) +L_{X}(\circ )(Z, [Y,W]).
\end{align*}
Our second claim follows this relation and Theorem \ref{main}.

{\it iii)} It is straightforward to check that a sum of eventual
identities on the factors gives an eventual identity on the
product $(M_{1}\times M_{2}, \circ , e_{1}+e_{2}).$ The converse
is more involved and goes as follows (a similar argument has been
used for the decomposition of Euler vector fields on product
$F$-manifolds, see Theorem 2.11 of \cite{hert}). Let $\mathcal E$
be an eventual identity on $(M_{1}\times M_{2}, \circ ,
e_{1}+e_{2})$ and define ${\mathcal E}_{k} := e_{k}\circ {\mathcal
E}$ for $k\in \{ 1,2\} .$ From (\ref{product-mult}) ${\mathcal
E}_{k}$ is tangent to $M_{k}$ at any point of $M_{1}\times M_{2}$.
We will show that $\mathcal E_{1}$ is a vector field on $M_{1}$ (a
similar argument shows that ${\mathcal E}_{2}$ is a vector field
on $M_{2}$). For this, let $Z$ be a vector field on $M_{2}$. Note
that
\begin{equation}\label{i}
L_{{\mathcal E}_{1}}(\circ )(Z, e_{2}) = {\mathcal E}
\circ L_{e_{1}}(\circ )(Z, e_{2}) +e_{1}\circ L_{\mathcal E}(\circ )
(Z, e_{2})=0
\end{equation}
because $L_{e_{1}}(\circ )=0$ (easy check) and
$$
e_{1}\circ L_{\mathcal E}(\circ ) (Z, e_{2})=e_{1}\circ [e,
{\mathcal E}]\circ Z\circ e_{2}=0
$$
where we used condition (\ref{char}) on $\mathcal E$ and
$e_{1}\circ e_{2}=0$. From (\ref{i}) and $Z= Z\circ e_{2}$ we get
$$
[{\mathcal E}_{1}, Z]=L_{\mathcal E_{1}}(Z\circ e_{2})= [{\mathcal
E}_{1}, Z]\circ e_{2} + Z\circ [{\mathcal E}_{1}, e_{2}].
$$
It follows that $[{\mathcal E}_{1}, Z]$ is tangent to $M_{2}$ at
any point of $M_{1}\times M_{2}$. This holds for any vector field
$Z$ on $M_{2}$ and hence ${\mathcal E}_{1}$ is a vector field on
$M_{1}.$ Similarly, ${\mathcal E}_{2}$ is a vector field on
$M_{2}$. Since $\mathcal E$ is invertible on $(M, \circ
,e_{1}+e_{2})$, ${\mathcal E}_{1}$ is invertible on $(M,
\circ_{1}, e_{1})$ and ${\mathcal E}_{2}$ is invertible on $(M,
\circ_{2}, e_{2}).$ From
$$
[e, {\mathcal E}] = [e_{1}, {\mathcal E}_{1}] + [e_{2}, {\mathcal
E}_{2}]
$$
and
$$
L_{\mathcal E}(\circ )(X, Y) = [e, {\mathcal E}]\circ X\circ
Y,\quad \forall X, Y\in {\mathcal X}(M)
$$
we get
$$
L_{{\mathcal E}_{k}}(\circ_{k})(X, Y) = [e_{k}, {\mathcal
E}_{k}]\circ X\circ Y,\quad \forall X, Y\in {\mathcal
X}(M_{k}),\quad k\in \{ 1,2\} ,
$$
i.e. ${\mathcal E}_{k}$ is an eventual identity on the
$F$-manifold $(M_{k}, \circ_{k}, e_{k}).$ Our claim follows.

\end{proof}

\noindent By a result of Hertling \cite{hert}, any $F$-manifold
locally decomposes into a product of irreducible $F$-manifolds.
The decomposition of eventual identities on product $F$-manifolds
into sums of eventual identities on the factors gives a
compatibility between our duality for $F$-manifolds with eventual
identities and Hertling's decomposition of $F$-manifolds, as
follows.

\begin{thm}Let $(M, \circ , e)$ be an $F$-manifold with irreducible
decomposition
\begin{equation}\label{product}
(M, \circ ,e )\cong (M_{1}, \circ_{1}, e_{1})\times\cdots \times
(M_{l}, \circ_{l}, e_{l})
\end{equation}
near a point $p\in M$ and let $\mathcal E$ be an eventual identity on $(M,
\circ , e)$. Consider the decomposition
\begin{equation}\label{decompev}
{\mathcal E} = {\mathcal E}_{1}+\cdots + {\mathcal E}_{l}
\end{equation}
of $\mathcal E$ into a sum of eventual identities ${\mathcal
E}_{k}$ on the factors. Let $(M, *, {\mathcal E}, e)$ be the dual
of $(M, \circ , e, {\mathcal E})$ and $(M_{k}, *_{k}, {\mathcal
E}_{k},e_{k})$ the dual of $(M_{k}, \circ_{k}, e_{k}, {\mathcal
E}_{k})$, for any $1\leq k\leq l$. Then
\begin{equation}\label{fin-dec}
(M, *, {\mathcal E})\cong (M_{1}, *_{1}, {\mathcal
E}_{1})\times\cdots \times (M_{l}, *_{l}, {\mathcal E}_{l})
\end{equation}
is the irreducible decomposition of the $F$-manifold $(M, *,
{\mathcal E})$ near $p$.
\end{thm}

\begin{proof} The decomposition (\ref{decompev}) was proved in
Proposition \ref{combine} {\it iii)}. The decomposition (\ref{fin-dec})
follows from (\ref{product}) and (\ref{decompev}).
\end{proof}

We end this Section with some more remarks and examples of
eventual identities.

\begin{rem}\label{remarks}{{\it i)} Condition (\ref{char}) which characterizes eventual
identities is equivalent to the apparently weaker condition
\begin{equation}\label{weak}
L_{\mathcal E}(\circ ) (X, Y ) = {v}\circ X\circ
Y,\quad\forall X, Y\in {\mathcal X}(M),
\end{equation}
for a vector field ${v}$. Indeed, if in relation (\ref{weak}) we
replace $X$ and $Y$ by $e$ we get ${v} = L_{\mathcal E}(\circ
)(e,e) .$ On the other hand,
$$
L_{\mathcal E}(\circ )(e,e) = [{\mathcal E},e\circ e] -
2[{\mathcal E},e]\circ e = [e, {\mathcal E}]
$$
and hence ${v}= [e, {\mathcal E}]$, as in (\ref{char}).
In particular, any invertible Euler vector field $E$ of weight $d$ is an
eventual identity and $[e, E]= de.$\\

{\it ii)} If $\mathcal E$ is an eventual identity on an
$F$-manifold $(M, \circ , e)$, then
\begin{equation}\label{euler-lie}
[{\mathcal E}^{n}, {\mathcal E}^{m}] = (m-n) {\mathcal
E}^{m+n-1}\circ [e, {\mathcal E}],\quad \forall m,n\in \mathbb{Z}.
\end{equation}
The proof is by induction. When $\mathcal E$ is Euler and $m,
n\geq 0$, (\ref{euler-lie}) was proved in \cite{manin-carte} (see
Theorem 5.6); when $n=-1$ and $m=0$ (\ref{euler-lie}) was
proved in Lemma \ref{pregatitoare}.\\

{\it iii)} Let $(M, \circ, e)$  be a semi-simple $F$-manifold with
canonical coordinates  $(u^{1}, \cdots , u^{n})$, i.e.
$$
\frac{\partial}{\partial u^{i}}\circ \frac{\partial}{\partial
u^{j}} = \delta_{ij} \frac{\partial}{\partial u^{j}}, \quad\forall
i, j
$$
and
$$
e = \frac{\partial~}{\partial u^{1}}+\cdots +
\frac{\partial~}{\partial u^{n}}.
$$
Any eventual identity is of the form
$$
{\mathcal E} = f_{1}\frac{\partial~}{\partial u^{1}} + \cdots +
f_{n}\frac{\partial~}{\partial u^{n}},
$$
where $f_{i}$ are smooth non-vanishing functions depending only on
$u^{i}.$\\

{\it iv)} Here is an example considered in \cite{hert}, when the
multiplication is not
semi-simple. Let $M:= \mathbb{R}^{2}$ with multiplication defined
by
$$
\frac{\partial}{\partial x^{1}}\circ\frac{\partial}{\partial
x^{i}}=\frac{\partial}{\partial x^{i}},\quad
\frac{\partial}{\partial x^{2}}\circ \frac{\partial}{\partial
x^{2}}=0, \quad i\in \{ 1,2\} .
$$
It can be checked that $\circ$ defines an $F$-manifold structure
and any eventual identity is of the form
$$
{\mathcal E} = f_{1}\frac{\partial}{\partial x^{1}} +
f_{2}\frac{\partial}{\partial x^{2}},
$$
where $f_{1}=f_{1}(x^{1})$ depends only on $x^{1}$ and is
non-vanishing.}

\end{rem}

\section{Eventual identities and compatible metrics}\label{F-riem}

The two metrics $g$ and $\tilde{g}$ on a Frobenius manifold have
the important property that they form a flat pencil, that is, the
metric $g_{\lambda}^{*}:= g^{*}+\lambda\tilde{g}^{*}$ is flat, for
all values of $\lambda\,.$ This condition results, via the
Dubrovin-Novikov theorem, to a bi-Hamiltonian structure. What is
important in this construction is not the flatness of the metrics
but their compatibility. Curved metrics can, via Ferapontov's
extension of the Dubrovin-Novikov theorem, define (non-local)
Hamiltonian structures but it is the compatibility of two such
metrics that will ensure a (non-local) bi-Hamiltonian structure.
In this Section we construct compatible pair of metrics on
$F$-manifolds with eventual identities.

We begin by recalling basic definitions and results on compatible
pair of metrics. First we fix the conventions we will use in this
and the following Sections.

\begin{notations}{\rm Let $g$ and $\tilde{g}$ be two metrics on a manifold $M$, with
associated pencil of inverse metrics $g_{\lambda}^{*}:=
g^{*}+\lambda\tilde{g}^{*}$ (assumed to be non-degenerate for any
$\lambda$). We denote by $g:TM\rightarrow T^{*}M$, $X\rightarrow
g(X)$ and $g^{*}:T^{*}M\rightarrow TM$, $\alpha \rightarrow
g^{*}(\alpha )$ the isomorphisms defined by raising and lowering
indices using $g$ and similar notations will be used for the
isomorphisms between $TM$ and $T^{*}M$ defined by $\tilde{g}$ and
$g_{\lambda}$. To simplify notations we shall often denote by
$X^{\flat}=\tilde{g}(X)$ the dual $1$-form of a vector field $X$
with respect to $\tilde{g}$ (it is important to note that
$X^{\flat}$ is the dual $1$-form using $\tilde{g}$ and not $g$,
since the metrics $g$ and $\tilde{g}$ will not play symmetric
roles). The Levi-Civita connections of $g$, $g_{\lambda}$ and
$\tilde{g}$ will be denoted by $\nabla$, $\nabla^{\lambda}$ and
$\tilde{\nabla}$ respectively; $R^{g}$, $R^{\lambda}$ and
$R^{\tilde{g}}$ and will denote the curvatures of $g$,
$g_{\lambda}$ and $\tilde{g}$.}
\end{notations}

\begin{defn} i) A pair $(g, \tilde{g})$ is called almost compatible
if
$$
g_{\lambda}^{*}(\nabla^{\lambda}_{X}\alpha )= g^{*}(\nabla_{X}\alpha )+
\lambda\tilde{g}^{*}(\tilde{\nabla}_{X}\alpha )
$$
for any $X\in {\mathcal X}(M)$, $\alpha\in\Omega^{1}(M)$ and
$\lambda$ constant.\\

ii) A pair $({g}, \tilde{g})$ is called compatible if
$(g,\tilde{g})$ are almost compatible and
\begin{equation}
g_{\lambda}^{*}(R^{\lambda}_{X, Y}\alpha ) = g^{*}(R^{g}_{X,
Y}\alpha ) +\lambda \tilde{g}^{*}( {R}^{\tilde{g}}_{X, Y}\alpha )
\end{equation}
for any $X, Y\in {\mathcal X}(M)$, $\alpha\in \Omega^{1}(M)$ and $\lambda$
constant.
\end{defn}

According to \cite{mok} (see also \cite{l-i} for a shorter proof)
the metrics $(g, \tilde{g})$ are almost compatible if and only if
the Nijenhuis tensor of $A:= g^{*}\tilde{g}\in\mathrm{End}(TM)$,
defined by
$$
N_{A}(X, Y) = -[AX, AY] +A\left( [AX, Y] + [X, AY]\right) -
A^{2}[X, Y],\quad X,Y\in{\mathcal X}(M)
$$
is identically zero. Moreover, according to Theorem 3.1 of
\cite{l-i}, if $(g, \tilde{g})$ are almost compatible then
$(g,\tilde{g})$ are compatible if and only if one of the following
equivalent conditions holds:
\begin{equation}\label{g1}
g^{*}(\tilde{\nabla}_{Y}\alpha - \nabla_{Y}\alpha ,
\tilde{\nabla}_{X}\beta -\nabla_{X}\beta ) =
g^{*}(\tilde{\nabla}_{X}\alpha - \nabla_{X}\alpha ,
\tilde{\nabla}_{Y}\beta -\nabla_{Y}\beta )
\end{equation}
or
\begin{equation}\label{g2}
\tilde{g}^{*}(\tilde{\nabla}_{Y}\alpha - \nabla_{Y}\alpha ,
\tilde{\nabla}_{X}\beta -\nabla_{X}\beta ) =
\tilde{g}^{*}(\tilde{\nabla}_{X}\alpha - \nabla_{X}\alpha ,
\tilde{\nabla}_{Y}\beta -\nabla_{Y}\beta ) ,
\end{equation}
for any vector fields $X, Y\in {\mathcal X}(M)$ and $1$-forms
$\alpha , \beta\in \Omega^{1}(M).$\\

We now turn to $F$-manifolds and we show in Proposition \ref{p}
bellow that an eventual identity on an $F$-manifold together with
a (multiplication) invariant metric determine a pair of almost
compatible metrics.
A metric $\tilde{g}$ on an $F$-manifold $(M, \circ , e )$ is
called invariant if
$$
\tilde{g}(X\circ Y, Z) = \tilde{g}(X, Y\circ Z),\quad\forall X, Y,
Z\in {\mathcal X}(M)
$$
or
$$
\tilde{g}(X, Y) = \epsilon (X\circ Y).
$$
where $\epsilon =\tilde{g}(e)$ is the coidentity. Thus $\tilde{g}$
is uniquely determined by the coidentity
$\epsilon\in\Omega^{1}(M)$ and invariant metrics on $(M, \circ , e
)$ are in bijective correspondence with $1$-forms on $M$.

\begin{prop}\label{p} Let $(M, \circ , e, \tilde{g}, {\mathcal E})$ be an $F$-manifold
together with an invariant metric $\tilde{g}$ and eventual
identity $\mathcal E$. Define a new metric $g$ by
\begin{equation}\label{newg}
{g}(X, Y) = \tilde{g}({\mathcal E}^{-1}\circ X, Y),\quad\forall X,
Y\in {\mathcal X}(M).
\end{equation}
Then $(g,\tilde{g})$ are almost compatible.\end{prop}

\begin{proof} From (\ref{newg}),
$$
g^{*}\tilde{g}(X) = {\mathcal E}\circ X,\quad \forall X\in TM
$$
Using the $F$-manifold condition (\ref{fman}) together with the
characterization (\ref{char}) of eventual identities, we get:

\begin{align*}
N_{{\mathcal E}\circ}(X, Y)&= -L_{{\mathcal E}\circ X}({\mathcal
E}\circ Y) +{\mathcal E}\circ \left( L_{X}({\mathcal E}\circ
Y)-L_{Y}({\mathcal E}\circ X)\right)
- {\mathcal E}^{2}\circ [X, Y]\\
&=-[{\mathcal E}\circ X, {\mathcal E}]\circ Y - [{\mathcal E}\circ
X, Y]\circ {\mathcal E} -
L_{{\mathcal E}\circ X}(\circ )({\mathcal E}, Y)\\
&+{\mathcal E}\circ\left( [X, {\mathcal E}]\circ Y +{\mathcal E}
\circ [X, Y] +L_{X}(\circ ) ({\mathcal E}, Y)
- [Y, {\mathcal E}]\circ X\right)\\
& - {\mathcal E}^{2}\circ [Y, X] - {\mathcal E}\circ L_{Y}(\circ )
({\mathcal E}, X)-{\mathcal E}^{2}\circ [X, Y]\\
&= L_{\mathcal E}({\mathcal E}\circ X) \circ Y+ L_{Y} ({\mathcal
E}\circ X)\circ {\mathcal E} -{\mathcal E}\circ L_{X}(\circ )
({\mathcal E}, Y)\\
& - X\circ L_{\mathcal E}(\circ ) ({\mathcal E}, Y) +{\mathcal
E}\circ Y
\circ[X, {\mathcal E}] +{\mathcal E}^{2}\circ [X, Y]\\
&+{\mathcal E}\circ L_{X}(\circ ) ({\mathcal E}, Y) - {\mathcal
E}\circ X
\circ [Y, {\mathcal E}] - {\mathcal E}^{2}\circ [Y, X]\\
& -{\mathcal E}\circ L_{Y}(\circ ) ({\mathcal E}, X)-{\mathcal
E}^{2}
\circ [X, Y]\\
&= L_{\mathcal E}(\circ )({\mathcal E}, X)\circ Y -
L_{\mathcal E}(\circ )({\mathcal E}, Y)\circ X\\
&= [e, {\mathcal E}]\circ{\mathcal E}\circ  (X\circ Y - Y\circ X)
=0,
\end{align*}
for any vector fields $X, Y\in {\mathcal X}(M)$. Our claim
follows.
\end{proof}

When the $F$-manifold $(M, \circ , e)$ is semi-simple, the pair
$(g, \tilde{g})$ of Proposition \ref{p} is semi-simple as well
and, being almost compatible,
$(g,\tilde{g})$ is automatically compatible \cite{mok,l-i}.
Without the semi-simplicity assumption, the pair
$(g,\tilde{g})$ is not always compatible. We are going to show
that $(g,\tilde{g})$ is compatible (without the semi-simplicity
assumption), provided that the coidentity associated to
$\tilde{g}$ is closed. To simplify terminology we introduce the
following definition.

\begin{defn}\label{defi} An almost Riemannian $F$-manifold
is an $F$-manifold $(M,
\circ , e,\tilde{g})$ together with an invariant metric
$\tilde{g}$ such that the coidentity $\epsilon \in\Omega^{1}(M)$
defined by
$$
\epsilon (X) := \tilde{g}(e, X),\quad\forall X\in TM
$$ is closed.
\end{defn}

There is a result of Hertling \cite{hert}, which states that the
closeness of the coidentity $\epsilon$ on an $F$-manifold $(M,
\circ , e,\tilde{g})$ with invariant metric $\tilde{g}$ is
equivalent with the total symmetry of the $(4,0)$-tensor field
\begin{equation}\label{sym}
(\tilde{\nabla}\circ )(X, Z, Y, V):=
\tilde{g}(\tilde{\nabla}_{X}(\circ )(Z, Y), V),
\end{equation}
or to the symmetry in the first two arguments (the symmetry in the
last three arguments being a consequence of the invariance of
$\tilde{g}$).

\begin{thm}\label{compth}
Let $(M, \circ , e, \tilde{g}, {\mathcal E})$ be an almost
Riemannian $F$-manifold with eventual identity $\mathcal E$.
Define a new metric $g$ by
$$
{g}(X, Y) = \tilde{g}({\mathcal E}^{-1}\circ X, Y),\quad\forall X,
Y\in {\mathcal X}(M).
$$
Then $(g,\tilde{g})$ are compatible.
\end{thm}

\begin{proof}
From Proposition \ref{p}, the metrics $(g, \tilde{g})$ are almost
compatible. To prove that $(g, \tilde{g})$ are compatible, it is
enough to show that (\ref{g2}) is satisfied (see our comments
above). The Koszul formula for the Levi-Civita $\nabla$ of $g$
translated to $T^{*}M$ gives
\begin{equation}\label{l1}
2g^{*}(\nabla_{Y}\alpha , \beta ) = - g^{*}(i_{Y}d\beta , \alpha )
+g^{*}(i_{Y}d\alpha , \beta ) + Yg^{*}(\alpha ,\beta )
-g([g^{*}\alpha , g^{*}\beta ],Y),
\end{equation}
where $\alpha , \beta\in \Omega^{1}(M)$ and $Y\in {\mathcal
X}(M).$ A similar expression holds for the Levi-Civita connection
$\tilde{\nabla}$ of $\tilde{g}$ on $T^{*}M$:
\begin{equation}\label{l2}
2\tilde{g}^{*}(\tilde{\nabla}_{Y}\alpha , \beta ) = -
\tilde{g}^{*}(i_{Y}d\beta , \alpha ) +\tilde{g}^{*}(i_{Y}d\alpha ,
\beta ) + Y\tilde{g}^{*}(\alpha ,\beta )
-\tilde{g}([\tilde{g}^{*}\alpha ,\tilde{g}^{*}\beta ],Y).
\end{equation}
Combining (\ref{l1}) and (\ref{l2}) and using that $(g,
\tilde{g})$ are almost compatible we get, by the argument
of Proposition 5.10 of \cite{l-i},
\begin{equation}\label{flat}
2g^{*}(\nabla_{Y}X^{\flat} -\tilde{\nabla}_{Y}X^{\flat},
Z^{\flat}) = (L_{\mathcal E}\tilde{g})(X\circ Y, Z)
+\tilde{g}(([e, {\mathcal E}]\circ X-2\tilde{\nabla}_{X}{\mathcal
E} )\circ Y,Z)
\end{equation}
where $X^{\flat},Z^{\flat}\in\Omega^{1}(M)$ correspond to
$X,Z\in{\mathcal X}(M)$ using the duality defined by $\tilde{g}$.
Now, for a vector field $V$, define a $1$-form $(L_{\mathcal
E}\tilde{g})(V)$ by
$$
(L_{\mathcal E}\tilde{g})(V)(Z):= (L_{\mathcal E}\tilde{g})(V, Z)
,\quad\forall Z\in {\mathcal X}(M).
$$
With this notation,
$$
(L_{\mathcal E}\tilde{g})(X\circ Y, Z) = (L_{\mathcal
E}\tilde{g})(X\circ Y)(Z).
$$
Since $L_{\mathcal E}\tilde{g}$ is multiplication invariant (this
follows by taking the Lie derivative with respect to $\mathcal E$
of $\tilde{g}(X\circ Y, Z) = \tilde{g}(X, Y\circ Z)$ and using
condition (\ref{char}) on $\mathcal E$), we obtain
\begin{equation}\label{flat1}
(L_{\mathcal E}\tilde{g})(X\circ Y) = X^{\flat}\circ
Y^{\flat}\circ (L_{\mathcal E}\tilde{g})(e)
\end{equation}
where $\circ$ is the induced multiplication on $T^{*}M$, obtained
by identifying $TM$ with $T^{*}M$ using $\tilde{g}$. Denoting
$\alpha := X^{\flat}$, from (\ref{flat}) and (\ref{flat1}) we get
\begin{equation}\label{new}
2(\nabla_{Y}\alpha - \tilde{\nabla}_{Y}\alpha )= Y^{\flat}\circ
{\mathcal E}^{-1,\flat} \circ\left( ((L_{\mathcal E}\tilde{g})(e)
+ [e, {\mathcal E}]^{\flat})\circ \alpha  -
2\tilde{\nabla}_{\tilde{g}^{*}\alpha}{\mathcal E^{\flat}}\right) .
\end{equation}
Since $\tilde{g}$ is invariant, $\tilde{g}^{*}$ is also invariant
(with respect to $\circ$ on $T^{*}M$) and relation (\ref{new})
implies that (\ref{g2}) is satisfied. Being almost compatible, the
metrics $(g,\tilde{g})$ are compatible.
\end{proof}

We end this Section by making some comments on Theorem
\ref{compth}. Similar results where proved in \cite{l-i}, with the
almost Riemannnian $F$-manifold replaced by a weak
$\mathcal F$-manifold $(M, \circ , e ,\tilde{g}, E)$, i.e. the
multiplication $\circ$ on $TM$ is commutative associative with
unity field $e$, $\tilde{g}$ is an invariant metric,
$E$ is an invertible Euler vector field which is also conformal-Killing
with respect to $\tilde{g}$
and the weak symmetry condition
\begin{equation}
(\tilde{\nabla }\circ )(E, Z, Y,V) = (\tilde{\nabla}\circ )(Z, E,
Y, V),\quad\forall  Y, Z,V \in {\mathcal X}(M)
\end{equation}
holds; in general, $\circ$ does not satisfy the integrability
condition (\ref{fman}), so a weak $\mathcal F$-manifold is not
always an $F$-manifold. We are going to show that a weak $\mathcal
F$-manifold which is also an $F$-manifold is an almost Riemannian
$F$-manifold. Thus, in the setting of $F$-manifolds, Theorem
\ref{compth} extends the statement about the compatibility of
metrics in Theorem 5.8 of \cite{l-i}, by replacing the Euler
vector field with an eventual identity.

\begin{lem} Let $(M, \circ , e ,{\mathcal E},\tilde{g})$ be an $F$-manifold
together with an invertible vector field $\mathcal E$ and
invariant metric $\tilde{g}$. Assume the weak symmetry
condition
\begin{equation}
(\tilde{\nabla }\circ )({\mathcal E}, Z, Y,V) =
(\tilde{\nabla}\circ )(Z, {\mathcal E}, Y, V),\quad\forall  Y, Z,V
\in {\mathcal X}(M)
\end{equation}
holds. Then $(M, \circ , e ,\tilde{g})$ is an almost Riemannian
$F$-manifold.
\end{lem}

\begin{proof}
We need to show that the coidentity $\epsilon =\tilde{g}(e)$ is
closed. It is known that on any $F$-manifold $(M, \circ ,
e,\tilde{g})$ with multiplication $\circ$, unity field $e$,
invariant metric $\tilde{g}$ and coidentity $\epsilon$, the tensor
fields $\tilde{\nabla}\circ$ and $d\epsilon$ are related by the
following identity (see the proof of Theorem 2.15 of \cite{hert}):
\begin{equation}\label{cheie}
2(\tilde{\nabla}\circ) (X, Z, Y,V) -2(\tilde{\nabla}\circ )(Z, X,
Y,V) = d\epsilon (Y\circ Z, X\circ V) - d\epsilon (X\circ Y,
Z\circ V) .
\end{equation}
Taking $X:= {\mathcal E}$ in (\ref{cheie}) and
using our hypothesis we get
\begin{equation}\label{de}
d\epsilon ({\mathcal E}\circ Y, Z\circ V) = d\epsilon  (Y\circ Z,
{\mathcal E}\circ V).
\end{equation}
With $Z:=e$, (\ref{de}) becomes
\begin{equation}\label{e}
d\epsilon ({\mathcal E}\circ Y, V) = d\epsilon (Y, {\mathcal
E}\circ V).
\end{equation}
Replacing in (\ref{e}) $V$ by $V\circ Z$ and using again
(\ref{de}) we get
\begin{equation}\label{inv}
d\epsilon (Y, {\mathcal E}\circ V\circ Z)= d\epsilon ({\mathcal
E}\circ Y, V\circ Z ) = d\epsilon (Y\circ Z, {\mathcal E}\circ V).
\end{equation}
Since ${\mathcal E}$ is invertible, relation (\ref{inv}) is
equivalent to
\begin{equation}\label{rep}
d\epsilon (Y, Z\circ V) = d\epsilon (Y\circ Z, V), \quad\forall Y,
Z, V\in {\mathcal X}(M),
\end{equation}
i.e. $d\epsilon$ is multiplication invariant. Being skew-symmetric,
$d\epsilon =0.$ Our claim follows.

\end{proof}

\section{Duality and Riemannian $F$-manifolds}\label{dual}

Riemannian $F$-manifolds were first introduced in the literature
in \cite{ita}. In this Section we prove that the class of
Riemannian $F$-manifolds is preserved by the duality between
$F$-manifolds with eventual identities. In the next Section we
apply this result to the theory of integrable systems.

\begin{defn} A Riemannian $F$-manifold is
an $F$-manifold $(M, \circ , e, \tilde{g})$
together with an invariant metric $\tilde{g}$ such that:\\

i) the coidentity $\epsilon=\tilde{g}(e)\in\Omega^{1}(M)$
is closed, i.e. $(M, \circ , e, \tilde{g})$  is an
almost Riemannian $F$-manifold.\\

ii) the curvature condition
\begin{equation}\label{riem}
Z\circ R^{\tilde{g}}(V, Y)(X) + Y\circ R^{\tilde{g}}(Z, V)(X) +
V\circ R^{\tilde{g}}(Y, Z)(X)=0,
\end{equation}
is satisfied, for any $X, Y, Z,V\in {\mathcal X}(M).$
\end{defn}

Our main result in this Section is the following Theorem.

\begin{thm}\label{r} Let $(M, \circ , e, \tilde{g}, {\mathcal E})$ be an $F$-manifold
with invariant metric $\tilde{g}$ and eventual identity
$\mathcal E$. Define a second metric $g$ by
\begin{equation}\label{f}
{g}(X, Y)= \tilde{g}({\mathcal E}^{-1}\circ X, Y),\quad\forall X,
Y\in {\mathcal X}(M)
\end{equation}
and let $(M, *, {\mathcal E},e)$ be the dual of  $(M, \circ , e,
{\mathcal E}).$ Then $(M, \circ , {e}, \tilde{g})$ is a Riemannian
$F$-manifold if and only if $(M, *, {\mathcal E}, g)$ is a
Riemannian $F$-manifold.
\end{thm}

\begin{proof} From (\ref{f}), the coidentities of $(M, \circ , e, \tilde{g})$
and $(M, *, {\mathcal E}, g)$ coincide. Thus $(M, \circ , {e},
\tilde{g})$ is an almost Riemannian $F$-manifold if and only if
$(M, *, {\mathcal E}, g)$ is an almost Riemannian $F$-manifold.

Assume now that $(M, \circ , e, \tilde{g})$ is a Riemannian
$F$-manifold. By our comments from the previous Section, the
tensor field $\tilde{\nabla}\circ$ is totally symmetric. With the
conventions from the proof of Theorem \ref{compth}, the total
symmetry of $\tilde{\nabla }\circ$ and relation (\ref{new}),
together with an easy curvature computation show that the
curvatures of $g$ and $\tilde{g}$ on $T^{*}M$ are related by
\begin{equation}\label{legat}
R^{g}(X, Y)(\alpha)= R^{\tilde{g}}(X, Y)(\alpha ) + Q (\alpha , Y)
\circ X^{\flat} - Q(\alpha , X)\circ Y^{\flat},
\end{equation}
where
$$
Q(\alpha , X):= {\mathcal S}({\mathcal S}(\alpha )\circ X^{\flat})
-\tilde{\nabla}_{X}({\mathcal S})(\alpha ),\quad\forall\alpha\in T^{*}M,
\quad\forall X\in TM
$$
and
\begin{equation}\label{salpha}
{\mathcal S}(\alpha ) :=\frac{1}{2}{\mathcal E}^{-1,\flat}
\circ\left( ((L_{\mathcal E}\tilde{g})(e) + [e, {\mathcal
E}]^{\flat})\circ \alpha  -
2\tilde{\nabla}_{\tilde{g}^{*}\alpha}{\mathcal E^{\flat}}\right) .
\end{equation}
(Recall that $TM$ and $T^{*}M$ are identified using $\tilde{g}$ and
$\circ$ above denotes the induced multiplication on $T^{*}M$).
Since $(M, \circ , e, \tilde{g})$ is a Riemannian $F$-manifold,
relation (\ref{riem}) holds. Translated to $T^{*}M$, it gives
\begin{equation}\label{doi}
Z^{\flat}\circ R^{\tilde{g}}(V, Y)(\alpha ) + Y^{\flat}\circ
R^{\tilde{g}}(Z, V)(\alpha ) + V^{\flat}\circ R^{\tilde{g}}(Y,
Z)(\alpha )=0,
\end{equation}
for any vector fields $Y$, $Z$ and $V$ and covector $\alpha .$
Using (\ref{legat}), relation (\ref{doi}) becomes
\begin{equation}\label{doi1}
Z^{\flat}\circ R^{{g}}(V, Y)(\alpha ) + Y^{\flat}\circ R^{{g}}(Z,
V)(\alpha ) + V^{\flat}\circ R^{{g}}(Y, Z)(\alpha )=0.
\end{equation}
Take in (\ref{doi1}) $\alpha := g(X)$. Note that
$$
Z^{\flat}\circ R^{g}(V, Y)(\alpha ) = Z^{\flat}\circ g(R^{g}(V,
Y)(X))= Z^{\flat}\circ {\mathcal E}^{-1,\flat}\circ R^{g}(V,
Y)(X)^{\flat}
$$
and similarly for $Y^{\flat}\circ R^{g}(Z, V)(\alpha )$ and
$V^{\flat}\circ R^{g}(Y, Z)(\alpha )$. On $TM$ relation (\ref{doi1}) becomes
\begin{equation}
{\mathcal E}^{-1}\circ \left(Z\circ  R^{{g}}(V, Y)(X) + Y\circ
R^{{g}}(Z, V)(X)+ V\circ R^{{g}}(Y, Z)(X)\right) =0
\end{equation}
for any vector fields $X, Y, Z, V$, or
\begin{equation}
Z* R^{{g}}(V, Y)(X) + Y * R^{{g}}(Z, V)(X) + V* R^{{g}}(Y,
Z)(X)=0.
\end{equation}
We proved that $(M, *, {\mathcal E}, g)$ is a Riemannian
$F$-manifold. Our claim follows.
\end{proof}

\section{Applications to integrable
systems}\label{integrable}

There is a  close relationship between $F$-manifolds and the
theory of integrable systems of hydrodynamic type. In particular
we draw together various results of \cite{ita} into the following
theorem.

\begin{thm}
Consider an almost Riemannian  $F$-manifold $(M,\circ,e,{\tilde{g}})\,.$ If $\tilde X$ and $\tilde Y$ are two vector fields which
satisfy the condition
\begin{equation}
({\tilde\nabla}_Z {\tilde X})\circ V = ({\tilde\nabla}_V {\tilde
X})\circ Z \qquad \forall\,V\,,Z \in {\mathcal X}(M)
\label{flowcondition}
\end{equation}
then the associated flows
\begin{eqnarray*}
U_t  & = & {\tilde X} \circ U_x \,,\\
U_\tau & = & {\tilde Y} \circ U_x
\end{eqnarray*}
commute. Moreover, for arbitrary vector fields $Y\,,V\,,Z
\in{\mathcal X}(M)$ the identity
\[
Z \circ R^{\tilde g}(V,Y) ({\tilde X}) + V \circ R^{\tilde g}(Y,Z)
({\tilde X}) + Y \circ R^{\tilde g}(Z,V)({\tilde X})  = 0
\]
holds for any solution ${\tilde X}$ of (\ref{flowcondition}).
\end{thm}
By twisting solutions ${\tilde X}$ of (\ref{flowcondition}) by an eventual identity one may derive the dual, or twisted, version of the
above theorem.

\begin{lem} Let $(M, \circ , e, \tilde{g})$
be an almost Riemannian $F$-manifold and
${\tilde X}\in {\mathcal X}(M)$ a vector field such that
\begin{equation}\label{lorentz}
\tilde{\nabla}_{Y}{\tilde X} \circ V = \tilde{\nabla}_{V}{\tilde X}\circ Y,\quad\forall
Y, V\in {\mathcal X}(M).
\end{equation}
Let $\mathcal E$ be an eventual identity on $(M, \circ , e)$ and
$(M, *, {\mathcal E}, g)$ the dual almost Riemannian $F$-manifold,
like in Theorem \ref{r}. Then $X={\tilde X}\circ \mathcal E$
satisfies the dual equation
\begin{equation}\label{lorentz-dual}
({\nabla}_{Y}X) * V = ({\nabla}_{V}X) * Y,\quad\forall Y, V\in
{\mathcal X}(M).
\end{equation}
\end{lem}

\begin{proof}
 Recall, from relation (\ref{new}), that
\begin{equation}\label{new2}
\nabla_{Y}\alpha = \tilde{\nabla}_{Y}\alpha + {\mathcal S}(\alpha
)\circ Y^{\flat},\quad\forall Y\in TM
\end{equation}
where ${\mathcal S}(\alpha )$ is given by (\ref{salpha}), as usual
$Y^{\flat}=\tilde{g}(Y)$ and $\circ$ is the induced multiplication
on $T^{*}M$, obtained by identifying $TM$ with $T^{*}M$ using
$\tilde{g}.$ In (\ref{new2}) let $\alpha := {\tilde X}^{\flat}=
g({\tilde X}\circ {\mathcal E}).$ Relation (\ref{new2}) becomes
\begin{equation}\label{new1}
g(\nabla_{Y}({\tilde X}\circ {\mathcal E})) = \tilde{\nabla}_{Y}
{\tilde X}^{\flat} + Y^{\flat}\circ {\mathcal S}({\tilde
X}^{\flat})
\end{equation}
Applying $\tilde{g}^{*}$ to (\ref{new1}) and using
$(\tilde{g}^{*}g)(X) = {\mathcal E}^{-1}\circ X$ we get
\begin{equation}\label{useful}
\nabla_{Y}({\tilde X}\circ {\mathcal E}) ={\mathcal E}\circ
\tilde{\nabla}_{Y}{\tilde X} +{\mathcal E}\circ Y\circ
\tilde{g}^{*}( {\mathcal S}({\tilde X}^{\flat})) .
\end{equation}
From (\ref{useful}) we get
$$
\nabla_{Y}({\tilde X}\circ{\mathcal E}) *V =
\tilde{\nabla}_{Y}{\tilde X}\circ V +Y\circ V \circ
\tilde{g}^{*}({\mathcal S}({\tilde X}^{\flat})),
$$
which, from (\ref{lorentz}), is symmetric in $Y$ and $V$.
Relation (\ref{lorentz-dual}) is satisfied.
\end{proof}

Thus we obtain dual flow equations
\begin{eqnarray*}
U_t  & = & X * U_x \,,\\
U_\tau & = & Y * U_x
\end{eqnarray*}
from vector fields ${\tilde X}\,,{\tilde Y} \in{\mathcal X}(M)$ satisfying (\ref{flowcondition}) by twisting
by an eventual identity. Moreover by Theorem \ref{r} the dual curvature condition also holds.

This duality, or twisting, by an eventual identity gives a geometric form of certain well-known arguments
from the theory of integrable systems of hydrodynamic type which originate in the work of Tsarev.
Recall that in the semi-simple case the basic equation $U_t = {\tilde X} \circ U_x$ reduces to diagonal form
\[
u^i_t = {\tilde X}^i({\bf u}) u^i_x
\]
so the components of ${\tilde X}$ become the characteristic velocities of the quasilinear system. Equation
(\ref{flowcondition}) reduces to Tsarev's equation
\begin{equation}
\frac{\partial~}{\partial u^i} \log\sqrt{{\tilde g}_{jj}} = \frac{\partial_i {\tilde X}^j}{{\tilde X}^i-{\tilde X}^j}\,, \qquad i\neq j\,.
\label{Tsarev}
\end{equation}
The integrability conditions for this system form the so-called semi-Hamiltonian conditions, which in turn are the coordinate
form of (\ref{riem}).

Solutions of (\ref{Tsarev}) possess a functional freedom: if
${\tilde{g}}_{ii}({\bf u})$ is a solution so is
${\tilde{g}}_{ii}({\bf u})/f_i(u^i)\,.$ This functional freedom
can now be reinterpreted, via Remark \ref{remarks} {\it iii)} on
the form of eventual identities in the semi-simple case, as the
dual version of the theory. Also since the $f_i$ are arbitrary,
one may replace it by $f_i \rightarrow f_i + \lambda$ for any
constant $\lambda\,.$ Thus one recovers the pencil property
$g^*_\lambda = g^* + \lambda {\tilde g}^*$ and hence, by
Proposition \ref{p}, a compatible pair of metrics and (non-local)
bi-Hamiltonian structures (this last stage, from almost compatible
to compatible being automatic in the semi-simple case).

In applications, where one is interested in finding bi-Hamiltonian structures for a specific system of equations, one tries to
find a suitable eventual identity so that the metric $g$ has simple curvature properties, such as flatness or constant curvature.
If flat one arrives, via the original Dubrovin-Novikov theorem, at a local Hamiltonian structure. The simplest case is where both metrics
are flat, and hence form a flat pencil and a local bi-Hamiltonian structure. With extra conditions one can arrive at a Frobenius manifold
\cite{dubrovin2}.

\section{Duality and $tt^{*}$-geometry}\label{ttstar-sect}

An holomorphic $F$-manifold is a complex manifold $M$ together with an
associative, commutative, with unity multiplication
$\circ$ on the sheaf of holomorphic vector fields, satisfying
the $F$-manifold condition (\ref{fman}). Euler vector fields, identities,
eventual identities etc are holomorphic and are defined
like in the smooth case.
In particular,
our characterization of eventual identities developed in Theorem \ref{main}
holds also in the holomorphic setting.

In the same framework like in Sections \ref{F-riem} and
\ref{dual}, we add structures - hermitian metrics and real
structures - on an holomorphic $F$-manifold $(M, \circ ,e)$  and
we study their behaviour under twisting with an eventual identity.
We assume that these structures are compatible with the
multiplication $\circ$, i.e. in the terminology of \cite{simpson}
they form harmonic Higgs bundles or $DChk$-structures and we
determine necessary and sufficient conditions on the eventual
identity such that the resulting dual structures are compatible in
the same way. Harmonic Higgs bundles and $DChk$-structures are
part of the so called CV-structures, introduced for the first time
by Cecotti and Vafa in \cite{cecotti} and further studied in the
literature, see \cite{hert1}, \cite{sabbah}.

First we fix our conventions in the holomorphic setting.

\begin{notations}{\rm In this Section $M$ will denote a
complex manifold, considered as a smooth manifold together with an
integrable complex structure $J$. Its real tangent bundle will be
denoted $TM.$ The sheaf of smooth real vector fields on $(M, J)$
will be denoted as always by ${\mathcal X}(M)$, the sheaf of
vector fields of type $(1,0)$ by ${\mathcal T}^{1,0}_{M}$ and the
sheaf of holomorphic vector fields by ${\mathcal T}_{M}.$ A
multiplication on the holomorphic tangent bundle $T^{1,0}M$ will
be trivially extended to the complexified bundle $T_{\mathbb{C}}M
= TM\otimes\mathbb{C}.$}
\end{notations}

Following \cite{dubrovin}, \cite{simpson}, \cite{hert} we give
the following definition, which recalls basic notions from the theory
of $tt^{*}$-geometry.

\begin{defn} i) A pair $(\tilde{g}, \tilde{h})$ formed  by a
complex bilinear, non-degenerate symmetric form $\tilde{g}$ and a
hermitian metric $\tilde{h}$ on $T^{1,0}M$ is called compatible if
the Chern connection $\tilde{D}$ of the holomorphic hermitian
vector bundle $(T^{1,0}M, \tilde{h})$ preserves $\tilde{g}$, i.e.
$\tilde{D}\tilde{g}=0.$\\

ii) Let $\tilde{h}$ be a hermitian metric and $\circ$ a
commutative, associative, multiplication with unity field $e$, on
$T^{1,0}M.$ Define a Higgs field $\tilde{C}\in \Omega^{1,0}(M,
\mathrm{End}(T^{1,0}M))$ by
$$
\tilde{C}_{X}Y:= X\circ Y.
$$
The hermitian metric $\tilde{h}$ on the Higgs bundle $(T^{1,0}M,
\tilde{C})$ is called harmonic (and $(T^{1,0}M, \tilde{C},
\tilde{h})$ is a harmonic Higgs bundle) if $\tilde{C}_{X}Y\in
{\mathcal T}_{M}$, for any $X, Y\in {\mathcal T}_{M}$ and the
$tt^{*}$-equations
\begin{equation}
(\partial^{\tilde{D}}\tilde{C})_{X, Y}:=
\tilde{D}_{X}( \tilde{C}_{Y}) -
\tilde{D}_{Y}( \tilde{C}_{X})
-\tilde{C}_{[X, Y]}=0
\end{equation}
and
\begin{equation}
R^{\tilde{D}}_{X, \bar{Y}}
+ [\tilde{C}_{X}, \tilde{C}^{\flat}_{\bar{Y}}]=0
\end{equation}
are satisfied, for any $X, Y\in{\mathcal T}^{1,0}_{M}.$ Above
$R^{\tilde{D}}$ denotes the curvature of the Chern connection
$\tilde{D}$ of $(T^{1,0}M, \tilde{h})$ and $\tilde{C}^{\flat}$ is
the adjoint of $\tilde{C}$ with respect to $\tilde{h}$, i.e.
$$
\tilde{h}(\tilde{C}_{X}Y, Z)= \tilde{h}(Y, \tilde{C}^{\flat}_{\bar{X}}Z),
\quad\forall Y, Z\in T^{1,0}M,\quad\forall X\in T_{\mathbb{C}}M.
$$

iii) Let $(T^{1,0}M, \tilde{C}, \tilde{h})$ be a harmonic
Higgs bundle  and $\tilde{k}$ a real
structure on $T^{1,0}M$ such that the complex bilinear form
$$
\tilde{g}(X, Y):= \tilde{h}(X, \tilde{k}Y)
$$
on $T^{1,0}M$ is symmetric and (multiplication) invariant. The
data $(T^{1,0}M, \tilde{C}, \tilde{h},\tilde{k})$ is called a
$D\tilde{C}\tilde{h}\tilde{k}$-structure if the pair $(\tilde{g},
\tilde{h})$ is compatible.
\end{defn}

We remark that a harmonic Higgs bundle
$(T^{1,0}M, \tilde{C}, \tilde{h})$ has an associated pencil of flat
connections
\begin{equation}
\tilde{D}^{z}:= \tilde{D} +\frac{1}{z}\tilde{C} +z\tilde{C}^{\flat}.
\end{equation}
The flatness property of this pencil encodes the entire geometry of the
harmonic Higgs bundle \cite{hert}.\\

For the remaining part of this Section we fix an
$F$-manifold $(M, \circ , e)$ together with an eventual identity
$\mathcal E$, hermitian metric $\tilde{h}$, and real
structure $\tilde{k}$ on $T^{1,0}M$ such that the complex bilinear
form
$$
\tilde{g}(X, Y) := \tilde{h}(X, \tilde{k}Y)
$$
on $T^{1,0}M$ is symmetric and invariant. Let
\begin{equation}\label{dual-multi}
X*Y:= X\circ Y\circ {\mathcal E}^{-1}
\end{equation}
be the dual multiplication, with associated Higgs field
$$
C_{X}Y:= X\circ Y\circ {\mathcal E}^{-1}.
$$
Assume that the inverse $\mathcal E^{-1}$ has a square root
${\mathcal E}^{-1/2}$
and define a new hermitian metric
\begin{equation}\label{h-def}
{h}(X, Y):= \tilde{h}({\mathcal E}^{-1/2}\circ X, {\mathcal E}^{-1/2}\circ Y)
\end{equation}
and a new real structure
$$
k(X) := {\mathcal E}^{1/2}\circ \tilde{k}({\mathcal E}^{-1/2}\circ X)
$$
on $T^{1,0}M.$
It is straightforward to check that
\begin{equation}\label{g}
g(X, Y):= h(X, kY)= \tilde{g}( {\mathcal E}^{-1/2}\circ X,
{\mathcal E}^{-1/2}\circ Y).
\end{equation}
In particular, $g$ is symmetric, complex bilinear and invariant.

While in the smooth case it was not immediately clear that
compatibility is preserved under twisting with eventual
identities, the analogous statement in the holomorphic setting
comes for free (and in fact holds under the weaker assumption that
$\mathcal E$ is holomorphic and invertible, not necessarily an
eventual identity).

\begin{lem} If the pair $(\tilde{g}, \tilde{h})$ is compatible, then also
the pair $(g, h)$ is compatible. \end{lem}

\begin{proof} From (\ref{h-def}), the  Chern connections
$D$ and $\tilde{D}$ of $(T^{1,0}M, h)$
and $(T^{1,0}M, \tilde{h})$ respectively are related by
\begin{equation}\label{pull-chern}
D_{X}Z:= {\mathcal E}^{1/2}\circ \tilde{D}_{X}( {\mathcal E}^{-1/2}
\circ Z),\quad\forall X\in {\mathcal X}(M),\quad Z\in
{\mathcal T}^{1,0}_{M}.
\end{equation}
From (\ref{g}) and (\ref{pull-chern}), $\tilde{D}\tilde{g}=0$ if
and only if $Dg=0$.
\end{proof}

Note that if $M$ is a Frobenius manifold with Euler vector field $E$ then the choice
${\mathcal E}=E$ results in a compatible pair $(g, h)$ with certain special
properties. The metric $g$ is the intersection form of the manifold, and hence is flat.
Thus there exists a distinguished coordinate system of so-called flat coordinates in which
the components of $g$ are constant. The metric $h$ is then a natural hermitian metric
defined on the complement of the classical discriminant $\Sigma$ of the manifold.

\begin{thm}\label{ttstar} i) Assume that
$\partial^{\tilde{D}}\tilde{C}=0.$
Then $\partial^{D}C =0$ if and only if for any
$X,Y, Z\in {\mathcal T}^{1,0}_{M}$,
\begin{equation}\label{herm1}
\tilde{D}_{X} ({\mathcal E}\circ Y\circ Z) -
\tilde{D}_{Y}({\mathcal E}\circ X\circ Z) =
{\mathcal E}\circ \left( \tilde{D}_{X}(Y\circ Z) -
\tilde{D}_{Y}(X\circ Z)\right)
\end{equation}

ii) Assume that for any $X, Y\in {\mathcal T}^{1,0}_{M}$,
\begin{equation}
R^{\tilde{D}}_{X, \bar{Y}}
+ [\tilde{C}_{X}, \tilde{C}^{\flat}_{\bar{Y}}]=0.
\end{equation}
Then the same relation holds with $\tilde{D}$ replaced by $D$,
$\tilde{C}$ replaced by $C$ and $\tilde{C}^{\flat}$ replaced by
the adjoint $C^{\flat}$ of $C$ with respect to $h$ if and only if,
for any $X,Y\in T^{1,0}M$,
\begin{equation}\label{herm2}
[\tilde{C}_{X}, \tilde{k}\tilde{C}_{Y}\tilde{k}] =
[\tilde{C}_{{\mathcal E}^{-1}\circ X}, \tilde{k}
\tilde{C}_{{\mathcal E}^{-1}
\circ Y}\tilde{k}].
\end{equation}

iii) If $(T^{1,0}M, \tilde{C},\tilde{h})$ is a harmonic
Higgs bundle (respectively, $(T^{1,0}M, \tilde{C},\tilde{h}, \tilde{k})$ is a
$\tilde{D}\tilde{C}\tilde{h}\tilde{k}$-structure) then
$(T^{1,0}M, {C},{h})$ is a harmonic
Higgs bundle (respectively,  $(T^{1,0}M,{C},{h},{k})$ is a
${D}{C}{h}{k}$-structure) if and only if both (\ref{herm1}) and
(\ref{herm2}) are satisfied.

\end{thm}

\begin{proof}
We only need to check
that (\ref{herm1}) and (\ref{herm2}) are equivalent to the
$tt^{*}$-equations for $(D, C, C^{\flat})$, the other statements being
trivial from our previous considerations.

From a straightforward computation which uses
$\partial^{\tilde{D}}\tilde{C}=0$, for any $X, Y\in {\mathcal
T}^{1,0}_{M}$,
\begin{align*}
(\partial^{D}C)_{X, Y}&= \tilde{D}_{X} (\tilde{C}_{\mathcal E^{-1}})\tilde{C}_{Y}
+\tilde{C}_{\mathcal E^{1/2}}\tilde{D}_{X}(\tilde{C}_{\mathcal E^{-1/2}})
\tilde{C}_{\mathcal E^{-1}\circ Y} + {\tilde{C}}_{\mathcal E^{-1/2}\circ X}
\tilde{D}_{Y}(\tilde{C}_{\mathcal E^{-1/2}})\\
&- \tilde{D}_{Y} (\tilde{C}_{\mathcal E^{-1}})\tilde{C}_{X}
-\tilde{C}_{\mathcal E^{1/2}}\tilde{D}_{Y}(\tilde{C}_{\mathcal E^{-1/2}})
\tilde{C}_{\mathcal E^{-1}\circ X}
- {\tilde{C}}_{\mathcal E^{-1/2}\circ Y}
\tilde{D}_{X}(\tilde{C}_{\mathcal E^{-1/2}})
\\
\end{align*}
To simplify notations, define $T\in \mathrm{End}_{\mathbb{C}}(T^{1,0}M)$
by $T(X):= {\mathcal E}^{-1/2}\circ X$.
Therefore, $\partial^{D}C=0$ is equivalent with

\begin{align*}
\left(T(\tilde{D}_{X}T) + (\tilde{D}_{X}T)T+
T^{-1}(\tilde{D}_{X}T)T^{2}\right)\tilde{C}_{Y}
+ T\tilde{C}_{X}\tilde{D}_{Y}T\\
- \left(T(\tilde{D}_{Y}T) + (\tilde{D}_{Y}T)T +
T^{-1}(\tilde{D}_{Y}T)T^{2}\right)
\tilde{C}_{X}
 - T\tilde{C}_{Y}\tilde{D}_{X}T=0.
\end{align*}
On the other hand, applying the covariant derivative
$\tilde{D}_{Y}$ (for $Y\in {\mathcal T}^{1,0}_{M}$) to the
relation
\begin{equation}\label{comutativa}
T\tilde{C}_{X} = \tilde{C}_{X}T,\quad\forall X\in {\mathcal T}^{1,0}_{M},
\end{equation}
skew-symmetrizing in $X$ and $Y$ and using
$\partial^{\tilde{D}}\tilde{C}=0$,
we obtain
\begin{equation}\label{dc}
(\tilde{D}_{X}T)\tilde{C}_{Y} - (\tilde{D}_{Y}T)\tilde{C}_{X} =
\tilde{C}_{Y}(\tilde{D}_{X}T) -\tilde{C}_{X}(\tilde{D}_{Y}T).
\end{equation}
Using (\ref{dc}), the condition
$\partial^{D}C =0$ becomes equivalent to
$$
\tilde{D}_{X} (T^{2})\tilde{C}_{Y} =
\tilde{D}_{Y}(T^{2})\tilde{C}_{X},
$$
which, in turn, is equivalent to
(\ref{herm1}) (easy check). This proves claim {\it i)}.

For claim {\it ii)}, we need to prove that (\ref{herm2}) is
equivalent with the remaining $tt^{*}$-equation
$$
R^{{D}}_{X, \bar{Y}} + [{C}_{X}, {C}^{\flat}_{\bar{Y}}]=0,
\quad\forall X,Y\in T^{1,0}M.
$$
This follows from a straightforward computation which uses
$$
R^{D}_{X, Y}= \tilde{C}_{{\mathcal E}^{1/2}}R^{\tilde{D}}_{X, Y}
\tilde{C}_{\mathcal E^{-1/2}},\quad\forall X, Y\in TM
$$
together with
$$
\tilde{C}^{\flat}_{X} = \tilde{k}\tilde{C}_{\bar{X}}\tilde{k}
$$
and
$$
C^{\flat}_{X} = kC_{\bar{X}}k = \tilde{C}_{\mathcal
E^{1/2}}\tilde{k} \tilde{C}_{X\circ \mathcal
E^{-1}}\tilde{k}\tilde{C}_{\mathcal E^{-1/2}}
$$
for any $X\in T^{0,1}M.$
\end{proof}

We remark that condition (\ref{herm1}) on
the eventual identity is invariant under our duality of Theorem
\ref{main}. The following simple result holds.

\begin{prop} Let $(M, *, {\mathcal E}, e)$ be the dual of
$(M, \circ , e ,{\mathcal E})$. If the eventual identity
$\mathcal E$  of $(M, \circ , e)$ satisfies
\begin{equation}
\tilde{D}_{X} ({\mathcal E}\circ Y\circ Z) -
\tilde{D}_{Y}({\mathcal E}\circ X\circ Z) =
{\mathcal E}\circ \left( \tilde{D}_{X}(Y\circ Z) -
\tilde{D}_{Y}(X\circ Z)\right) ,
\end{equation}
then the eventual identity $e$ of $(M, *, {\mathcal E})$ satisfies
the dual condition
\begin{equation}
{D}_{X} (e* Y* Z) -
{D}_{Y}({e}* X* Z) =
e* \left( {D}_{X}(Y* Z) -
{D}_{Y}(X* Z)\right) ,
\end{equation}
for any $X, Y, Z\in {\mathcal T}^{1,0}_{M}.$
\end{prop}

\begin{proof} Straightforward computation, which uses
(\ref{dual-multi}) and (\ref{pull-chern}).
\end{proof}

\subsection{CV-structures and duality}

A CV-structure on the holomorphic tangent bundle of a complex
manifold $M$ is a $\tilde{D}\tilde{C}\tilde{h}\tilde{k}$-structure
together with two endomorphisms $\tilde{\mathcal U}$ and
$\tilde{\mathcal Q}$ of $T^{1,0}M$, satisfying some additional
compatibility conditions. In particular, the endomorphism
$\tilde{\mathcal Q}$ is hermitian with respect to $\tilde{h}$ and,
as it turns out, $\tilde{\mathcal U}=\tilde{C}_{E}$, where $E$ is
an Euler vector field of weight one of the underlying $F$-manifold
$(M, \circ ,e).$

It is immediately clear that CV-structures are not preserved by
our duality of $F$-manifolds with eventual identities. The reason
is that if $E$ is an invertible Euler vector field on an
$F$-manifold $(M, \circ ,e)$, then $e$ is not Euler for the dual
$F$-manifold $(M, *, {E}).$ With this motivation, in Section
\ref{s1} we define CV-structures in a weaker sense, with the Euler
vector field replaced by an eventual identity. In Section \ref{s2}
we prove that weak CV-structures so defined are preserved by our
duality of $F$-manifolds with eventual identities, provided that
the eventual identity satisfies conditions (\ref{herm1}) and
(\ref{herm2}) of Theorem \ref{ttstar}.

\subsubsection{Weak CV-structures}\label{s1}

We begin by recalling basic definitions and results about
CV-structures on the holomorphic tangent bundle of a complex
manifold. Our treatment of CV-structures follows closely
\cite{hert}, where more details and proofs can be found. It is
worth remarking that sometimes our conventions differ from those
used in \cite{hert}. While we use the generic notation $\tilde{C}$
for a Higgs field and $\tilde{C}^{\flat}$ for its adjoint with
respect to a hermitian metric, the general notation in \cite{hert}
for a Higgs field is $C$ and $\tilde{C}$ denotes its adjoint with
respect to a hermitian metric. Moreover, in our conventions
$\tilde{C}$ is related to the associated multiplication $\circ$ on
the tangent bundle by $\tilde{C}_{X}Y = X\circ Y$, while in
\cite{hert} $C_{X}Y = - X\circ Y$. Hopefully these differences
will not generate any confusion.

\begin{defn}\label{cv} A CV-structure is a
$\tilde{D}\tilde{C}\tilde{h}\tilde{k}$-structure $(T^{1,0}M,
\tilde{C}, \tilde{h}, \tilde{k})$ together with two endomorphisms
$\tilde{\mathcal U}$ and $\tilde{\mathcal Q}$ of $T^{1,0}M$ such
that the
following conditions hold:\\

i) for any $X\in T^{1,0}M$, $[\tilde{C}_{X}, \tilde{\mathcal U}]=0.$\\

ii) $\tilde{D}_{\bar{X}}\tilde{\mathcal U}=0$ for any $X\in
T^{1,0}M$, i.e. if $Z\in {\mathcal T}_{M}$ then also
$\tilde{\mathcal U}(Z)\in {\mathcal T}_{M}$.\\

iii) the $(1,0)$-part of $\tilde{D}\tilde{\mathcal U}$ has the
following expression:
\begin{equation}\label{u}\tilde{D}_{X}\tilde{\mathcal U} +
[\tilde{C}_{X}, \tilde{\mathcal Q}] -\tilde{C}_{X}=0,\quad\forall
X\in T^{1,0}M.
\end{equation}

iv) $\tilde{\mathcal Q}$ is hermitian with respect to $\tilde{h}$;
moreover, $\tilde{\mathcal{Q}} +\tilde{k}{\tilde{\mathcal
Q}}\tilde{k} =0$, or, equivalently, $\tilde{\mathcal Q}$ is
skew-symmetric with respect to complex bilinear form $\tilde{g}$
on $T^{1,0}M$, defined as usual by $\tilde{g}(X, Y) = \tilde{h}(X, \tilde{k}Y)$.\\

v) the $(1,0)$-part of $\tilde{D}\tilde{\mathcal Q}$ has the
following expression:
\begin{equation}\label{cond-q} \tilde{D}_{X}{ \tilde{\mathcal
Q}}-[\tilde{C}_{X}, \tilde{k}\tilde{\mathcal
U}\tilde{k}]=0,\quad\forall X\in T^{1,0}M.
\end{equation}

\end{defn}

Let $\circ$ be the multiplication on $T^{1,0}M$, related to the
Higgs field $\tilde{C}$ by $X\circ Y:=\tilde{C}_{X}Y$, for any $X
,Y\in {\mathcal T}^{1,0}_{M}$ and denote by $e\in {\mathcal
T}_{M}$ its unity vector field. Recall that $(M, \circ ,e)$ is an
$F$-manifold (this is a consequence of the $tt^{*}$-equation
$\partial^{\tilde{D}}\tilde{C}=0$, see Lemma 4.3 of \cite{hert}).
From {\it i)}, $\tilde{\mathcal U}$ is the multiplication by a
vector field $\mathcal E=\tilde{\mathcal U}(e)\in {\mathcal
T}^{1,0}M.$ Condition {\it ii)} together with $e\in {\mathcal
T}_{M}$ imply that $\mathcal E$ is holomorphic and condition
(\ref{u}) with $\tilde{\mathcal U} = \tilde{C}_{\mathcal E}$
implies that $\mathcal E$ is an Euler vector field of weight one
for $(M, \circ ,e)$
(again by Lemma 4.3 of \cite{hert}).\\

We now define the more general notion of weak CV-structures.

\begin{defn} A weak CV-structure is a
$\tilde{D}\tilde{C}\tilde{h}\tilde{k}$-structure $(T^{1,0}M,
\tilde{C}, \tilde{h}, \tilde{k})$ together with two endomorphisms
$\tilde{\mathcal U} = \tilde{C}_{\mathcal E}$ (where ${\mathcal
E}\in {\mathcal T}_{M}$) and $\tilde{\mathcal Q}$ of $T^{1,0}M$,
satisfying all conditions of Definition \ref{cv}, except that
(\ref{u}) is replaced by the weaker condition
\begin{equation}\label{u-weak}
\tilde{D}_{X}\tilde{\mathcal U} + [\tilde{C}_{X}, \tilde{\mathcal
Q}] -\tilde{C}_{[e, {\mathcal E}]}\tilde{C}_{X}=0,\quad\forall
X\in T^{1,0}M.
\end{equation}
\end{defn}

While a CV-structure determines a preferred Euler vector field on
the underlying $F$-manifold, a weak CV-structure determines a
vector field $\mathcal E$ which satisfies the weaker condition
(\ref{slaba}), see below. In particular, if $\mathcal E$ is
invertible, then $\mathcal E$ is an eventual identity.

\begin{lem}\label{inceput} Let $(M, \circ ,e)$ be an $F$-manifold and $\tilde{D}$
a connection on $T^{1,0}M$ such that
$\partial^{\tilde{D}}\tilde{C}=0$, where $\tilde{C}_{X}Y= X\circ
Y$ is the Higgs field. Let $\mathcal E$ be a vector field of type
$(1,0)$ on $M$.\\

i) Assume that
\begin{equation}\label{slaba}
L_{\mathcal E}(\circ )(X, Y) = [e, {\mathcal E}]\circ X\circ
Y,\quad\forall X, Y\in {\mathcal T}^{1,0}_{M}.
\end{equation}
Then
\begin{equation}\label{du}
\tilde{D}_{X}(\tilde{C}_{\mathcal E}) + [\tilde{C}_{X},
\tilde{D}_{\mathcal E} - L_{\mathcal E}] -\tilde{C}_{[e, {\mathcal
E}]}\tilde{C}_{X}=0,\quad\forall X\in {\mathcal T}^{1,0}_{M}.
\end{equation}

ii) Conversely, assume that
\begin{equation}\label{du2}
\tilde{D}_{X}(\tilde{C}_{\mathcal E}) + [\tilde{C}_{X}, {\mathcal
Q}] -\tilde{C}_{[e, {\mathcal E}]}\tilde{C}_{X}=0,\quad\forall
X\in {\mathcal T}^{1,0}_{M},
\end{equation}
for an endomorphism $\tilde{Q}$ of $T^{1,0}M$. Then $\mathcal E$
satisfies (\ref{slaba}) and $\tilde{Q}$ is equal to
$\tilde{D}_{\mathcal E}-L_{\mathcal E}$ up to addition with
$\tilde{C}_{Z}$, for $Z\in {\mathcal T}^{1,0}_{M}$.
\end{lem}

\begin{proof} Assume that (\ref{slaba}) holds. Then, for any $X\in {\mathcal
T}^{1,0}_{M}$,
\begin{align*}
\tilde{D}_{X} (\tilde{C}_{\mathcal E}) + [ \tilde{C}_{X},
\tilde{D}_{\mathcal E} - L_{\mathcal E}]= \tilde{D}_{X}
(\tilde{C}_{\mathcal E}) -\tilde{D}_{\mathcal E} (\tilde{C}_{X})
+[L_{\mathcal E}, \tilde{C}_{X}]\\
= \tilde{C}_{[X, {\mathcal E}]} +L_{\mathcal E}(X\circ ) =
\tilde{C}_{[e, {\mathcal E}]\circ X},
\end{align*}
where we used the $tt^{*}$-equation
${\partial}^{\tilde{D}}\tilde{C}=0$ and the condition
(\ref{slaba}). Our first claim follows. We now prove the second
claim. As already mentioned above, if $[e, {\mathcal E}]=e$ then
(\ref{du2}) implies that $\mathcal E$ is Euler of weight one.
Without this additional assumption, the same argument shows that
(\ref{du2}) implies (\ref{slaba}). Therefore, (\ref{du}) holds as
well and ${\mathcal Q}-\tilde{D}_{\mathcal E} + L_{\mathcal E}$
commutes with $\tilde{C}_{X}$ for any $X\in T^{1,0}M.$ Thus
${\mathcal Q}-\tilde{D}_{\mathcal E} + L_{\mathcal E}$ is the
multiplication by a vector field $Z\in {\mathcal T}^{1,0}_{M}.$

\end{proof}

The following Proposition provides a useful characterization of
weak CV-structures. A similar statement for CDV-structures already
appears in the literature (see Theorem 2.1 of \cite{lin}).

\begin{prop}\label{ajutatoare-fin}
Let $(T^{1,0}M, \tilde{C}, \tilde{h},\tilde{k})$ be a
$\tilde{D}\tilde{C}\tilde{h}\tilde{k}$-structure. Define
$\tilde{g}(X, Y) = \tilde{h}(X, \tilde{k}Y)$ as usual and let
$\mathcal E$ be an eventual identity of the underlying
$F$-manifold $(M, \circ ,e )$. Then
$(T^{1,0}M,\tilde{C},\tilde{h}, \tilde{k},
\tilde{\mathcal{U}}=\tilde{C}_{\mathcal E})$ extends to a weak
CV-structure (i.e. there is an endomorphism $\tilde{Q}$ of
$T^{1,0}M$ such that $(T^{1,0}M, \tilde{C}, \tilde{h}, \tilde{k},
\tilde{\mathcal U},\tilde{Q})$ is a weak CV-structure) if and only
if there is $Z\in {\mathcal T}_{M}$ such that
\begin{equation}\label{h}
L_{\mathcal E -\bar{\mathcal E}}(\tilde{h})(X, Y)= \tilde{h}(X,
Y\circ Z)- \tilde{h}(X\circ Z, Y),\quad\forall X, Y\in T^{1,0}M
\end{equation}
and
\begin{equation}\label{lieg}
L_{\mathcal E}(\tilde{g}) (X, Y)= -2\tilde{g}(X\circ Y,
Z),\quad\forall X, Y\in T^{1,0}M
\end{equation}
hold. Moreover, $Z$ is uniquely determined by (\ref{lieg}) and
\begin{equation}\label{tilde-q}
\tilde{Q} = \tilde{D}_{\mathcal E} -L_{\mathcal E} +\tilde{C}_{Z}.
\end{equation}
\end{prop}

\begin{proof}
Since $\mathcal E$ is an eventual identity,  Lemma \ref{inceput}
implies that any endomorphism $\tilde{Q}$ such that $(T^{1,0}M,
\tilde{h}, \tilde{k}, \tilde{\mathcal U},\tilde{Q})$ is a weak
CV-structure must be of the form (\ref{tilde-q}), with $Z\in
{\mathcal T}^{1,0}_{M}$, and such that the relations
\begin{equation}\label{h1}
\tilde{h}(\tilde{Q}(Y), V) = \tilde{h}(Y, \tilde{Q}(V))
\end{equation}
and
\begin{equation}\label{g1a}
\tilde{g}(\tilde{Q}(Y), V) +\tilde{g}(Y, \tilde{Q}(V))=0,
\end{equation}
hold, for any $Y, V\in {\mathcal T}^{1,0}_{M}.$ We will show that
(\ref{h}) and (\ref{lieg}) are equivalent with (\ref{h1}) and
(\ref{g1a}) respectively. Since $\tilde{D}$ is the Chern
connection of $(T^{1,0}M, \tilde{h})$, for any $X\in {\mathcal
T}_{M}$ and $Y, V\in {\mathcal T}^{1,0}_{M}$,
\begin{align*}
L_{X}(\tilde{h})(Y, V)&= X\tilde{h}(Y, V) - \tilde{h}(L_{X}Y, V) -
\tilde{h}(Y, L_{\bar{X}}V)\\
&= \tilde{h}((\tilde{D}_{X} -L_{X})(Y), V) - \tilde{h}(Y,
(\tilde{D}_{\bar{X}}-L_{\bar{X}})(V)).\\
\end{align*}
On the other hand, since $X$ is holomorphic and
$\tilde{D}^{(0,1)}= \bar{\partial}$, $L_{\bar{X}} =
\tilde{D}_{\bar{X}}$ on $T^{1,0}M$ and we obtain
\begin{equation}\label{withoutb}
L_{X}(\tilde{h})(Y, V) =\tilde{h}( (\tilde{D}_{X}-L_{X})Y,
V),\quad\forall Y, V\in {\mathcal T}^{1,0}_{M}.
\end{equation}
Similarly,
\begin{equation}\label{without}
L_{\bar{X}}(\tilde{h})(Y, V) =\tilde{h}(Y,
(\tilde{D}_{X}-L_{X})V),\quad\forall Y, V\in {\mathcal
T}^{1,0}_{M}.
\end{equation}
Relations (\ref{withoutb}) and (\ref{without}) with $X= {\mathcal
E}$ imply that (\ref{h}) is equivalent with (\ref{h1}). A similar
argument which uses
\begin{equation}\label{deriv-lie}
L_{X}(\tilde{g})(Y, Z) =\tilde{g}( (\tilde{D}_{X}-L_{X})Y, Z) +
\tilde{g}( Y,(\tilde{D}_{X}-L_{X})Z)
\end{equation}
shows that (\ref{lieg}) is equivalent with (\ref{g1a}).

Assume now that there is $Z\in {\mathcal T}^{1,0}_{M}$ (uniquely
determined, since $\tilde{g}$ is non-degenerate) such that both
(\ref{h}) and (\ref{lieg}) are satisfied and define an
endomorphism $\tilde{Q}$ of $T^{1,0}M$ by (\ref{tilde-q}). Then
$(T^{1,0}M, \tilde{C}, \tilde{h}, \tilde{k}, \tilde{\mathcal U},
\tilde{Q})$ is a weak CV-structure provided that relation
(\ref{cond-q}) is satisfied. We will show that (\ref{cond-q}) is
satisfied if and only if $Z$ is holomorphic. For this we make the
following computation: for any $X\in {\mathcal T}_{M}$,
\begin{equation}\label{computation}
\tilde{D}_{X}(\tilde{D}_{\mathcal E} - L_{\mathcal E})
-[\tilde{C}_{X}, \tilde{k}\tilde{C}_{\mathcal E}\tilde{k}] =
[\tilde{D}_{X},\tilde{D}_{\mathcal E} - L_{\mathcal E}]
+[\tilde{D}_{X}, \tilde{D}_{\bar{\mathcal E}}] = \tilde{D}_{[X,
{\mathcal E}]} - [\tilde{D}_{X}, L_{{\mathcal E}-\bar{\mathcal
E}}]
\end{equation}
where in the first equality we used
\begin{equation}\label{t-crv}
[\tilde{C}_{X}, \tilde{k}\tilde{C}_{\mathcal E}\tilde{k}] = -
R^{\tilde{D}}_{X, \bar{\mathcal E}} =
-[\tilde{D}_{X},\tilde{D}_{\bar{\mathcal E}}]
\end{equation}
(from the $tt^{*}$-equation and $[X, \bar{\mathcal E}]=0$) and in
the second equality we used $\tilde{D}_{\bar{\mathcal E}}=
L_{\bar{\mathcal E}}$, because $\mathcal E\in {\mathcal T}_{M}$,
and $[\tilde{D}_{X},\tilde{D}_{\mathcal E}]= \tilde{D}_{[X,
{\mathcal E}]}$, because the curvature of $\tilde{D}$ is of type
$(1,1)$. On the other hand, using (\ref{h}) and (\ref{lieg}) and
taking the Lie derivative of $\tilde{g}(X, Y) = \tilde{h}(X,
\tilde{k}Y)$ with respect to $\mathcal E$, we get
\begin{equation}
L_{{\mathcal E}-\bar{\mathcal E}}(\tilde{k}) =
\tilde{k}\tilde{C}_{Z} + \tilde{C}_{Z}\tilde{k}
\end{equation}
or, equivalently,
\begin{equation}\label{derivv-k}
L_{{\mathcal E}-\bar{\mathcal E}}(Y)= - \tilde{k} L_{{\mathcal
E}-\bar{\mathcal E}}(\tilde{k}Y) +\tilde{C}_{Z}Y
+\tilde{k}\tilde{\mathcal C}_{Z}\tilde{k}(Y),\quad\forall Y\in
{\mathcal T}^{1,0}_{M}.
\end{equation}
From (\ref{derivv-k}), relation (\ref{computation}) becomes
\begin{align*}
\tilde{D}_{X}(\tilde{D}_{\mathcal E} - L_{\mathcal E})
-[\tilde{C}_{X}, \tilde{k}\tilde{C}_{\mathcal E}\tilde{k}] &=
\tilde{D}_{[X, {\mathcal E}]} +
[\tilde{k}\tilde{D}_{\bar{X}}\tilde{k}, \tilde{k}L_{{\mathcal
E}-\bar{\mathcal E}}\tilde{k}] - [\tilde{D}_{X},
\tilde{C}_{Z}+\tilde{k}\tilde{C}_{Z}\tilde{k}]\\
& =\tilde{D}_{[X, {\mathcal E}]}+\tilde{k}[\tilde{D}_{\bar{X}},
L_{{\mathcal E}-\bar{\mathcal E}}]\tilde{k}
-\tilde{D}_{X}(\tilde{C}_{Z})-\tilde{k}\tilde{D}_{\bar{X}}(\tilde{C}_{Z})\tilde{k}\\
&= \tilde{D}_{[X, {\mathcal E}]}-\tilde{k}
\tilde{D}_{\overline{[X, {\mathcal E}]}}\tilde{k}
-\tilde{D}_{X}(\tilde{C}_{Z})-
\tilde{k}D_{\bar{X}}(\tilde{C}_{Z})\tilde{k}\\
& = -\tilde{D}_{X}(\tilde{C}_{Z})
-\tilde{k}\tilde{D}_{\bar{X}}(\tilde{C}_{Z}) \tilde{k}
\end{align*}
where we used $\tilde{D}_{X}(Y)=
\tilde{k}\tilde{D}_{\bar{X}}(\tilde{k}Y)$ (because
$\tilde{D}\tilde{k}=0$) and
$$
[\tilde{D}_{\bar{X}}, L_{\mathcal E}] = [L_{\bar{X}}, L_{\mathcal
E}] = L_{[\bar{X}, {\mathcal E}]}=0,
$$
(because $X, \mathcal E\in {\mathcal T}_{M}$). We deduce that
$$
\tilde{D}_{X} (Q) - [\tilde{C}_{X}, \tilde{k}\tilde{C}_{\mathcal
E}\tilde{k}] =
-\tilde{k}\tilde{D}_{\bar{X}}(\tilde{C}_{Z})\tilde{k}.
$$
Therefore, (\ref{cond-q}) is satisfied if and only if
$\tilde{D}_{\bar{X}}(\tilde{C}_{Z})=0$, for any $X\in {\mathcal
T}^{1,0}_{M}$, i.e. $Z$ is holomorphic. Our claim follows.

\end{proof}

\subsubsection{Weak CV-structures and duality}\label{s2}

Our aim in this Section is to prove the following result.

\begin{thm}\label{ttstar-weak} Let $(T^{1,0}M, \tilde{C}, \tilde{h}, \tilde{k})$ be
a $\tilde{D}\tilde{C}\tilde{h}\tilde{k}$-structure, $\mathcal E$
an eventual identity on the underlying $F$-manifold $(M, \circ
,e)$ and $\tilde{\mathcal U} := \tilde{C}_{\mathcal E}.$ Assume
that conditions (\ref{herm1}) and (\ref{herm2}) are satisfied and
let $(T^{1,0}M, {C}, {h}, {k})$ be the dual
${D}{C}{h}{k}$-structure, as in Theorem \ref{ttstar}. Then
$(T^{1,0}M, \tilde{C}, \tilde{h}, \tilde{k}, \tilde{\mathcal U} )$
extends to a weak CV-structure if and only if $(T^{1,0}M, {C},
{h}, {k}, {\mathcal U}:= C_{e})$ extends to a weak CV-structure.
\end{thm}

\begin{proof}
Assume that $(T^{1,0}M,\tilde{C}, \tilde{h}, \tilde{k},
\tilde{\mathcal{U}})$ extends to a weak CV-structure. In order to
apply Proposition \ref{ajutatoare-fin} we need to determine an
holomorphic vector field $Z$ such that both (\ref{h}) and
(\ref{lieg}) hold, with $\circ$ replaced by $*$, $\tilde{h}$
replaced by $h$ and $\tilde{g}$ replaced by $g$. Define
\begin{equation}\label{defn-z}
Z:= -(\tilde{D}_{e}e)\circ {\mathcal E} +\frac{1}{2}
L_{e}({\mathcal E})
\end{equation}
and notice that it is holomorphic: from the $tt^{*}$-equations and
$\tilde{D}^{(0,1)}=\bar{\partial}$, we get:
$$
\bar{\partial}_{\bar{X}}(\tilde{D}_{e}e)=\tilde{D}_{\bar{X}}\tilde{D}_{e}e
= R^{\tilde{D}}_{\bar{X}, e}e= [\tilde{C}_{e},
\tilde{k}\tilde{C}_{{X}}\tilde{k}]=0,\quad\forall X\in T^{1,0}M,
$$
because $e$ is holomorphic and $\tilde{C}_{e}$ is the identity
endomorphism. Therefore, $\tilde{D}_{e}e$ and hence also $Z$ is
holomorphic. We now prove that the relations
\begin{equation}\label{deriv-g}
L_{e}(g)(X, Y) = -2g(X*Z, Y),\quad\forall X, Y\in T^{1,0}M
\end{equation}
and
\begin{equation}\label{deriv-h}
L_{e-\bar{e}}(h)(X, Y) =  h(X, Y*Z) -h(X*Z, Y),\quad\forall X,
Y\in T^{1,0}M
\end{equation}
hold, where $*$ is the dual multiplication
\begin{equation}\label{gres}
X*Y= C_{X}{Y}=X\circ Y\circ {\mathcal E}^{-1},\quad\forall X, Y\in
T^{1,0}M.
\end{equation}
Taking the Lie derivative with respect to $e$ of the relation
$$
g(X, Y) = \tilde{g}( X\circ{\mathcal E}^{-1}, Y)
$$
and using (\ref{deriv-lie}) with $X:= e$, together with
$L_{e}(\circ )=0$ and
\begin{equation}\label{del}
(\tilde{D}_{e}-L_{e})(X) = (\tilde{D}_{e}e)\circ X,\quad\forall
X\in T^{1,0}M
\end{equation}
(relation (\ref{del}) is an easy consequence of the
$tt^{*}$-equation $\partial^{\tilde{D}}\tilde{C}=0$, for details
see Theorem 4.5 of \cite{hert}), we get:
\begin{align*}
L_{e}(g)(X, Y)&= L_{e}(\tilde{g}) (X\circ {\mathcal E}^{-1} , Y)
+\tilde{g}(X\circ L_{e}({\mathcal E}^{-1}), Y)\\
&=2\tilde{g}((\tilde{D}_{e}e)\circ {\mathcal E}^{-1}\circ X, Y)
+\tilde{g}( L_{e}({\mathcal E}^{-1}) \circ X, Y)\\
&= 2g((\tilde{D}_{e}e)\circ X, Y) +g({\mathcal E}\circ
L_{e}({\mathcal E}^{-1})\circ X, Y)\\
&= -2g(X*Z, Y),
\end{align*}
for any $X, Y\in {\mathcal T}^{1,0}_{M}$. Relation (\ref{deriv-g})
follows. A similar computation shows that (\ref{deriv-h}) holds as
well. From Proposition \ref{ajutatoare-fin}, $(T^{1,0}M, C, h, k,
{\mathcal U})$ extends to a weak CV-structure.
\end{proof}

\subsection{The semi-simple case}

Recall that a holomorphic $F$-manifold $(M,\circ)$ is called
semi-simple if there are local coordinates $(u^1\,,\ldots\,,u^m)$
on $M$ such that the multiplication $\circ$ is diagonal (see
Remark \ref{remarks} iii). In the restricted case where the
hermitian metric $\tilde h$ and real structure $\tilde k$ are also
diagonal (and note that in general they need not be diagonal) the
various conditions of Theorem \ref{ttstar} are automatically
satisfied. More precisely, we can state.

\begin{ex} Any eventual identity on a semi-simple $F$-manifold
$(M, \circ, \tilde{h}, \tilde{k})$ with hermitian metric and real
structure taking the form
$$
\frac{\partial}{\partial u^{i}}\circ \frac{\partial}{\partial
u^{j}} = \delta_{ij}\frac{\partial}{\partial u^{j}},\quad
\tilde{h}(\frac{\partial}{\partial u^{i}},\frac{\partial}{\partial
u^{j}})= H_{ii}\delta_{ij},\quad
\tilde{k}(\frac{\partial}{\partial u^{i}}) = k_{i}
\frac{\partial}{\partial u^{i}},
$$
(where $|k_{i}|=1$ and $H_{ii}>0$ for any $i$) automatically satisfies the conditions (\ref{herm1}) and
(\ref{herm2}).
\end{ex}

\begin{proof} We assume that the multiplication, hermitian metric
and real structure are defined as above in canonical coordinates.

Let $\mathcal E$ be an eventual identity, given by $\mathcal E=
\sum_{i=1}^{n}f_{i}\frac{\partial}{\partial u^{i}}.$ Recall that
$f_{i}$ depends on the variable $u^{i}$ only. We will check
(\ref{herm1}) for fundamental vector fields $X=
\frac{\partial}{\partial u^{i}}$, $Y=\frac{\partial}{\partial
u^{j}}$ ($i\neq j$) and $Z= \frac{\partial}{\partial u^{p}}.$
Since the multiplication is semi-simple, (\ref{herm1}) is clearly
satisfied if $p\notin\{ i, j\}$. If $p=i$ say, then (\ref{herm1})
becomes
$$
\tilde{D}_{\frac{\partial}{\partial u^{j}}}
(f_{i}\frac{\partial}{\partial u^{i}})= {\mathcal E}\circ
\tilde{D}_{\frac{\partial}{\partial
u^{j}}}(\frac{\partial}{\partial u^{i}}),
$$
or, since $f_{i}$ depends only on $u_{i}$ and $i\neq j$,
\begin{equation}\label{de-veri}
f_{i}\tilde{D}_{\frac{\partial}{\partial u^{j}}}
(\frac{\partial}{\partial u^{i}})= {\mathcal E}\circ
\tilde{D}_{\frac{\partial}{\partial
u^{j}}}(\frac{\partial}{\partial u^{i}}).
\end{equation}
On the other hand, since $\tilde{D}$ is the Chern connection of
$\tilde{h}$,
$$
\tilde{D}_{X} (\frac{\partial}{\partial u^{i}}) =
\partial_{X}\mathrm{log} (H_{ii}) \frac{\partial}{\partial
u^{i}},\quad\forall X\in T^{1,0}M,\quad\forall i.
$$
In particular, $\tilde{D}_{\frac{\partial}{\partial u^{j}}}
(\frac{\partial}{\partial u^{i}})$ is a multiple of
$\frac{\partial}{\partial u^{i}}$ and (\ref{de-veri}) follows. We
proved that relation (\ref{herm1}) holds. It remains to prove
relation (\ref{herm2}). From the definitions of the real structure
and multiplication in canonical coordinates, it can be checked
that for any $Y:= \sum_{i=1}^{n}Y^{i}\frac{\partial}{\partial
u^{i}}$, the composition $\tilde{k}\tilde{C}_{Y} \tilde{k}$ is the
multiplication by the vector
$\sum_{i=1}^{n}\overline{Y^{i}}\frac{\partial}{\partial u^{i}}$.
In particular, both sides of (\ref{herm2}) vanish. Our claim
follows.

\end{proof}

It should be pointed out that the equations (\ref{herm1}) and
(\ref{herm2}) place highly restrictive conditions on the various
structures and may, in general, have no solution (as happens for
some of the two-dimensional non-semi-simple examples in
\cite{Takahasi}). Just as almost-dual Frobenius manifolds satisfy
{\rm almost} all of the axioms of a Frobenius manifold, asking for
the twisted structures to satisfy the full $tt^{*}$ axioms may be
too restrictive a condition. However, the above example does show
that solutions in the semi-simple case - albeit in the subclass of
diagonal real and hermitian structures - do exist.

LIANA DAVID: Institute of Mathematics Simion Stoilow of the
Romanian Academy, Calea Grivitei no. 21, Sector 1, Bucharest,
Romania; E-mail address: liana.david@imar.ro\\

IAN A. B. STRACHAN: Department of Mathematics, University of
Glasgow, Glasgow G12 8QW, UK; E-mail address: i.strachan@maths.gla.ac.uk

\end{document}